\documentclass{amsart}

\usepackage{amssymb}
\usepackage{url}
\usepackage{amscd}
\usepackage{texdraw}

\theoremstyle{plain}

\newtheorem*{M}{Main Theorem}

\newtheorem{corollary}{Corollary}[section]
\newtheorem{lemma}{Lemma}[section]
\newtheorem{remark}{Remark}[section]
\newtheorem{proposition}{Proposition}[section]
\newtheorem{definition}{Definition}[section]
\input txdtools

\begin{document}
\title[Thurston type theorem for sub-hyperbolic rational maps]
{ Thurston type theorem for sub-hyperbolic rational maps}
\author{Gaofei Zhang and  Yunping Jiang}

\address{Department of  Mathematics \\ Nanjing University \\ 210093, P. R. China}
\email{zhanggf@hotmail.com}

\address{Department of Mathematics \\ Queens College of CUNY \\ Flushing,
NY 11367} \email{Yunping.Jiang@qc.cuny.edu}

\thanks{}


\maketitle

\begin{abstract}
In 1980's, Thurston established a combinatorial characterization for
post-critically finite rational maps. This criterion was then
extended by Cui, Jiang, and Sullivan  to sub-hyperbolic rational
maps. The goal of this paper is to present a new but simpler proof
of this result by adapting the argument in the proof of Thurston's
Theorem.
\end{abstract}

\section{Introduction}

Let $f:S^{2} \to S^{2}$ be an orientation-preserving branched
covering map of degree $d \ge 2$. We denote by $\deg_{x} f$ the
local degree of $f$ at $x$.  We will call
$$
\Omega_{f} = \{x\in S^{2}\:\big{|}\: \deg_{f}(x) \ge 2\}
$$
the critical set of $f$ and
$$
P_{f} = \overline{\bigcup_{k\ge 1}f^{k}(\Omega_{f})}.
$$
the post-critical set.  We say $f$ is post-critically finite if
$P_{f}$ is a finite set.  In 1980's, Thurston established a
combinatorial characterization for post-critically finite rational
maps.   The theorem says that if the associated orbifold
${\mathcal{O}}_{f}$ is hyperbolic, then   $f$ is combinatorially
equivalent to a rational map if and only if it has no Thurston
obstructions.  The basic idea of the proof is as follows.  Consider
the Teichm\"uller space $T_{f}$  modeled on $(S^{2}, P_{f})$. Then
$f$ induces an analytic  operator $\sigma_{f}: T_{f} \to T_{f}$.  It
turns out that the existence of a rational map which realizes $f$ is
equivalent to the existence of a fixed point of $\sigma_{f}$. The
proof is then reduced to showing that $\sigma_{f}$ is a strictly
contracting map. The reader may refer to \cite{DH} for a detailed
proof of this theorem.

A natural question is that to what extent, Thurston's theorem can be
extended to rational maps with infinitely many post-critical points.
It was proved by McMullen that having no Thurston obstruction is
essentially true for any rational map with a hyperbolic
orbifold\:---\:\:only trivial Thurston obstructions inside Siegel
disks or Herman rings may occur for a rational map with a hyperbolic
orbifold~\cite{Mc} .  In 1994, Cui, Jiang, and Sullivan established
a Thurston type theorem for sub-hyperbolic rational maps
(\cite{CJS}, see also \cite{CT},\cite{T}). The original proof of
Cui-Jiang-Sullivan's theorem is quite involved. The goal of this
paper is to give a new but simpler proof of this theorem by adapting
the argument used in the proof of Thurston's theorem.

Before we present this theorem, let us introduce some definitions
first. We say $f$ is geometrically finite if $P_{f}$ is an infinite
set but with finitely many accumulation points. Suppose that $f$ is
geometrically finite. Then it is not difficult to see that the
accumulation set of $P_{f}$ consists of finitely many periodic
cycles. We leave this to the reader as an exercise. Let $P'_{f}$
denote the set of all the accumulation points of $P_{f}$.

\begin{definition}\label{sub-hyperbolic}{\rm
Let $f: S^{2} \to S^{2}$ be a geometrically finite branched covering
map of degree $d \ge 2$. We say $f$ is a sub-hyperbolic
semi-rational branched covering  if
 for any  $a \in P'_{f}$ of period $p \ge 1$,
there is an open neighborhood $U$ of $a$, such that $f$ is
holomorphic in $U$, and moreover, if $\deg_{a} f^{p} = 1$, then
$$
f ^{ p}(z) = a + \lambda (z- a) + o(|z - a|) \hbox{ for } z \in U$$
where $0< |\lambda| < 1$ is some constant, and if $\deg_{a} f^{ p} =
k > 1$, then
$$
 f ^{p} (z) = a +  \alpha (z-a)^{k} +
o(|z-a|^{k}) \hbox{ for } z \in U
$$
where $\alpha \ne 0$ is some constant. }
\end{definition}

As in the post-critically finite case, one can define Thurston
obstructions for a sub-hyperbolic semi-rational branched covering
map $f$ in a similar way. If $\gamma$ is a simple closed curve in
$S^{2}\setminus P_{f}$, then the set $f^{-1}(\gamma)$ is a union of
disjoint simple closed curves. If $\gamma$ moves continuously, so
does each component of $f^{-1}(\gamma)$. A simple closed curve
$\gamma$ is non-peripheral if each component of $S^{2}\setminus
\gamma$ contains at least two points of $P_{f}$. Consider a
multi-curve
$$
\Gamma = \{\gamma_{1}, \cdots, \gamma_{n}\}
$$
of simple, closed, disjoint, non-homotopic, and non-peripheral
curves in $S^{2} \setminus P_{f}$.  We say that $\Gamma$ is
$f$-stable if for any $\gamma \in \Gamma$,  every non-peripheral
component of $f^{-1}(\gamma)$ is homotopic in $S^{2}\setminus P_{f}$
to an element of $\Gamma$.

For each $f$-stable multi-curve $\Gamma$, define a linear
transformation,
$$
f_{\Gamma}:{\Bbb R}^{\Gamma} \to {\Bbb R}^{\Gamma}
$$
as follows: let $\gamma_{i,j,\alpha}$ denote the components of
$f^{-1}(\gamma_{j})$ homotopic to $\gamma_{i}$ in $S^{2} \setminus
P_{f}$ and $d_{i,j,\alpha}$ be the degree of
$f|_{\gamma_{i,j,\alpha}}: \gamma_{i,j,\alpha} \to \gamma_{j}$.
Define
$$
f_{\Gamma}(\gamma_{j})= \sum_{i}\big{(}
\sum_{\alpha}\frac{1}{d_{i,j,\alpha}}\big{)}\gamma_{i}.
$$
Since the matrix of $f_{\Gamma}$ is non-negative, there exists a
largest eigenvalue $\lambda(\Gamma,f) \in {\Bbb R}_{+}$.  We say
that a multi-curve $\Gamma$ is a Thurston obstruction of $f$ if
$\lambda(\Gamma,f) \geq 1$.

\begin{definition}~\label{clhequivlent}
{\rm Suppose $f$ and $g$ are two sub-hyperbolic semi-rational
branched coverings. We say that they are
CLH-equivalent(combinatorially and locally holomorphically
equivalent) if there exist a pair of homeomorphisms $\phi: S^{2} \to
S^{2}$ and $\psi: S^{2} \to S^{2}$  such that
\begin{itemize}
\item $\psi$ is  isotopic to $\phi$ rel $\overline{P}_f$,
\item $\phi f=g\psi$,
\item $\phi|U_{f}=\psi|U_{f}$ is holomorphic on some open set $U_{f}\supset P_{f}'$.
\end{itemize}}
\end{definition}

Now let us state the Thurston type theorem for sub-hyperbolic
rational maps.
\begin{M}
Suppose $f$ is a sub-hyperbolic semi-rational branched covering.
Then $f$ is CLH-equivalent to a rational map $R$ if and only if $f$
has no Thurston obstructions. In this case, the rational map $R$ is
unique up to a M\"{o}bius conjugation of the Riemann sphere.
\end{M}
\begin{remark}{\rm
There are  branched covering maps of the sphere which are
geometrically finite and having no Thurston obstructions but are not
combinatorially equivalent to rational maps. For the construction of
such maps, see \cite{CJS2}.}
\end{remark}
 The proof of the "only if" part follows from a theorem of
McMullen(see Appendix B of \cite{Mc}). The main task of this paper
is to  prove the "if" part.

The essential difference between the post-critically finite case and
the sub-hyperbolic  case is that in the first case, the
post-critical set is a finite set and the Thurston pull back induces
an analytic operator defined on a finite-dimensional Teichm\"uller
space, while in the latter case, the post-critical set is an
infinite set and therefore, the induced operator is defined on an
infinite-dimensional Teichm\"uller space.   However, we observe in
this paper that, in both cases, the following bounded geometry
properties are similar. This allows  us to prove the latter case by
adapting the argument in the proof of the first case.

In the post-critically finite case, the base point of the
Teichm\"uller space is the Riemann sphere minus the set of finite
number of post-critical points. The branched covering induces a
pull-back operator on this Teichm\"uller space. Iterations of this
operator produce a sequence of sets of finite number of points in
the Riemann sphere. The bounded geometry in this case means that
there is a positive constant such that any two points in any element
of this sequence have spherical distance greater than or equal to
this constant.

In the sub-hyperbolic  case, the base point of the Teichm\"uller
space is the Riemann sphere minus the union of finitely many points
and topological disks. Iterations of the pull-back operator produce
a sequence of sets of finite number of points plus finite number of
disks in the Riemann sphere. The bounded geometry in this case means
that there is a positive constant such that in any element of this
sequence, the spherical distance between any two points,   any point
and any disk,  or any two disks is greater than or equal to this
constant; moreover, any disk in any element of this sequence
contains another round disk of radius greater than or equal to this
constant.

The paper is organized as follows. In \S2, we  prove the Shielding
Ring Lemma. The proof is elementary but it is crucial in our
construction of the Teichm\"uller space. In \S3, we  construct the
Teichm\"uller space $T_{f}$. In $\S4$, we  introduce the pull back
operator $\sigma_{f}: T_{f} \to T_{f}$. In \S5, we  introduce the
concept of bounded geometry. In \S6, we  prove that bounded geometry
implies the strictly contracting property of $\sigma_{f}$. In $\S7$,
we prove that no Thurston obstruction implies the bounded geometry.
This completes the proof of the Main Theorem.

\vskip20pt

\noindent {\bf Acknowledgement:} This work is based on part of the
first author's 2002-CUNY PH.D. thesis~\cite{Zh}.  The authors would
like to thank Professors Guizhen Cui, Dennis Sullivan, and Tan Lei
for many  conversations and help during this research.

%%%%%%%%%%%%%%%%%%%%%%%%%%%%%%%%%%%%%%%%%%%%%%%%%%%%%%%%%%%%%%%%%%%%%%%%%%%%
%%%%%%%%%%%%%%%%%%%%%%%%%%%%%%%%%%%%%%%%%%%%%%%%%%%%%%%%%%%%%%%%%%%%%%%%%%%%
%%%%%%%%%%%%%%%%%%%%%%%%%%%%%%%%%%%%%%%%%%%%%%%%%%%%%%%%%%%%%%%%%%%%%%%%%%%%

\section{Shielding Ring Lemma}

We say an open annulus $A$ is attached to an open topological disk
$D$ from the outside if $A$ and $D$ are disjoint but  $\partial D$
 is the inner  boundary component of the annulus $A$.
Then $\overline{D\cup A}$ is a larger closed disk.

Suppose that $f$ is a sub-hyperbolic semi-rational branched
covering. Let $P_f'=\{ a_{i}\}$. The main purpose of this section is
to prove the following lemma.

\begin{lemma}[Shielding Ring Lemma]~\label{srl}
There is a collection $\{ D_{i}\}$ of open  disks and a collection
of open  annuli $\{A_{i}\}$ such that
\begin{itemize}
\item $a_{i}\in D_{i}$,
\item $\overline{D}_{i}\cap \overline{D_{j}} =\emptyset$ for $i \ne j$,

\item for each $i$, $A_{i}$ is an annulus attaching $D_{i}$ from the outside
such that $\overline{A_{i}}\cap P_{f} = \emptyset$,
\item $f$ is holomorphic on $\overline{D_{i}}\cup A_{i}$,
\item every $f(A_{i})$ is contained in
some $D_{j}$.
\end{itemize}
\end{lemma}

\begin{proof}

Since $P_{f}'$ consists of finitely many   periodic cycles, it is
sufficient to find $D_{i}$ and $A_{i}$ for each periodic cycle.

Suppose
$$
\{ a_{1},\cdots, a_{p}\}
$$
is a periodic cycle in $P_{f}'$ such that
$$
f(a_{i}) =a _{i+1\pmod{p}}, \quad 1\leq i\leq p.
$$
This periodic cycle is either attracting or super-attracting. Let us
assume that we are in the attracting case. That is, we can find a
topological disk $W$ containing $a_{1}$ and a holomorphic
isomorphism $\phi: W \to \Delta$ such that $\phi \circ f^{p} \circ
\phi^{-1}: \Delta \to \Delta$ is equal to $\lambda z$ for some $0<
|\lambda|<1$. The super-attracting case can be treated in a similar
way by making minor changes.

For $0< r  < 1$, let ${\Bbb T}_{r} = \{z\:\big{|}\:|z| = r\}$ and
$\Delta_{r} = \{z \:\big{|}\: |z| < r\}$. Let $U_{r} =
\phi^{-1}({\Delta}_{r})$. Note that there are only countably many
$r$ such that
$$
\bigcup_{i \ge 0} f^{i}(\partial U_{r}) \cap  P_{f} \ne \emptyset.
$$
So we can take $0< a < 1$ such that
\begin{equation}\label{value-a}
\bigcup_{i \ge 0} f^{i}(\partial U_{a}) \cap  P_{f} \ne \emptyset.
\end{equation}

Let  $b = a + \epsilon < 1$ for some $\epsilon
> 0$ small. Let
$$
H = \{z\:\big{|}\:a < |z| < b\}.
$$
From (\ref{value-a}), it follows that by taking $\epsilon>0$ small,
we can assume
\begin{itemize}
\item[1.] $\bigcup_{0\le i \le p-1} f^{i}(\phi^{-1}(\overline{H})) \cap  P_{f} =
\emptyset$, and
\item[2.] $f^{p}(\phi^{-1}(\overline{H})) \subset U_{a}$.
\end{itemize}
Now divide the annulus $H$ into $p$ sub-annuli $H_{1}, \cdots,
H_{p}$ as follows. Take $a=r_{0}< r_{1}< r_{2}<\cdots<r_{p} = b$.
Let $H_{i} = \{z\:\big{|}\: r_{i-1} < |z|<r_{i}\}$. Let $E_{i} =
\phi^{-1}(H_{i})$. Define
$$
D_{1} =U_{a}\quad \hbox{and} \quad A_{1}=E_{1},
$$
and
$$
D_{2} = f(\overline{U_{a}}\cup E_{1}) \quad \hbox{and}\quad A_{2}
=f(E_{2}).
$$
For $3\leq i\leq p$,
$$
D_{i} = f^{i-1} (\overline{U_{a}} \cup \cup_{1\leq j\leq
i-2}\overline{E}_{j} \cup E_{i-1}) \quad \hbox{and} \quad A_{i}
=f^{i-1}(E_{i}).
$$

After we did for every periodic cycle in $P_{f}'$, we put those
disks and annuli together to get a collection  of open topological
disks $\{ D_{i}\}$ and a collection  of open  annuli $\{ A_{i}\}$.
By the construction, it is clear that they satisfy the properties in
Lemma~\ref{srl}. This completes the proof of Lemma~\ref{srl}.
\end{proof}

We call the disk  $ D_{i}$ in  Lemma~\ref{srl} a holomorphic disk
and the corresponding annulus $A_{i}$ a shielding ring.
\begin{remark}\label{analytic-boundary}{\rm
By our construction, the boundary of every $D_{i}$ is a
real-analytic curve.}
\end{remark}

%%%%%%%%%%%%%%%%%%%%%%%%%%%%%%%%%%%%%%%%%%%%%%%%%%%%%%%%%%%%%%%%%%%%%
%%%%%%%%%%%%%%%%%%%%%%%%%%%%%%%%%%%%%%%%%%%%%%%%%%%%%%%%%%%%%%%%%%%%%
%%%%%%%%%%%%%%%%%%%%%%%%%%%%%%%%%%%%%%%%%%%%%%%%%%%%%%%%%%%%%%%%%%%%%

\section{The Teichm\"uller space $T_{f}$}
  Let us now fix
a collection of holomorphic disks $\{ D_{i}\}$ and a collection of
shielding rings $\{ A_{i}\}$ for $f$. Let
$$
D_{f} =\cup_{i} D_{i}\quad \hbox{and}\quad P_{1} = P_f\setminus
D_{f}.
$$
By taking $D_{i}$ smaller, we may assume that $\#(P_{1})\geq 3$. We
may further assume that $\{0, 1, \infty\} \subset P_{1}$. Define
$$
Q_f = P_{1} \cup \overline{D}_{f} \quad \hbox{and} \quad X_{f}
=\partial Q_f = P_1 \cup \partial D_{f}.
$$

\begin{definition}\label{Teichmuller space}
{\rm The Teichm\"uller space $T_{f}$ is the Teichm\"uller space
modeled on $(S^{2} \setminus Q_{f}, X_{f})$. }
\end{definition}

The Teichm\"{u}ller space $T_{f}$  can be constructed as the space
of all the Beltrami coefficients defined  on $S^{2}\setminus Q_{f}$
module the following equivalent relation: let $\mu$ and $\nu$ be two
Beltrami coefficients defined on $S^{2}\setminus Q_{f}$ and let
$$
\phi_{\mu}: S^{2} \setminus Q_{f} \to S \hbox{ and } \phi_{v}:
S^{2}\setminus Q_{f}\to R
$$
be two quasiconformal homeomorphisms which solve the Beltrami
equations given by $\mu$ and $\nu$, respectively. we say $\mu$ and
$\nu$ are equivalent to each other if there exists a holomorphic
isomorphism $h: R \to S$ such that the map $\phi_{\mu}$ and $h\circ
\phi_{\nu}$ are isotopic to each other rel $X_{f}$, that is, there
is a continuous family of quasiconformal homeomorphisms $g_{t}:S^{2}
\setminus Q_{f} \to S, \:0\le t \le 1$,  such that
\begin{itemize}
\item[1.] $g_{0} = \phi_{\mu}$,
\item[2.] $g_{1} = h\circ \phi_{\nu}$,
\item[3.] $g_{t}(z) = \phi_{\mu}(z) = (h\circ \phi_{\nu})(z)$ for all $0 \le t \le
1$ and $z \in X_{f}$.
\end{itemize}

Now let us give a brief description of the relative background about
the Teichm\"{u}ller space $T_{f}$.   The reader may refer to
\cite{Ga} for more knowledge in this aspect.

Let $\mu$ be a Beltrami coefficient defined on $S^{2} \setminus
Q_{f}$. Let $$\phi_{\mu}: S^{2}\setminus Q_{f} \to  \phi_{\mu}(S^{2}
\setminus Q_{f})$$ be a quasiconformal homeomorphism which solves
the Beltrami equation given by $\mu$. Let
$$
{\rm M}_{u} = \{\xi(z) \frac{d\overline{z}}{dz} \: \big{|}\: \xi(z)
\hbox{ is measurable and } \|\xi\|_{\infty} < \infty\} $$  be  the
linear space of all the Beltrami differentials  defined on
$\phi_{\mu}(S^{2} \setminus Q_{f})$. Let
$$
{\rm A}_{u} = \{q(z) dz^{2} \: \big{|}\: q(z) \hbox{ is holomorphic
and } \int_{\phi_{\mu}(S^{2} \setminus Q_{f})} |q(z)| dz \wedge d
\overline{z} < \infty\}
$$
be the linear space of all the integrable holomorphic quadratic
differentials  defined on $\phi_{\mu}(S^{2} \setminus Q_{f})$.

A Beltrami differential $\xi(z) \frac{d\overline{z}}{dz} \in {\rm
M}_{\mu}$ is called $\emph{infinitesimally trivial}$ if
$$
\int_{\phi_{\mu}(S^{2} \setminus Q_{f})} \xi(z) q(z) dz \wedge d
\overline{z} = 0
$$
holds for all $q(z)dz^{2} \in {\rm A}_{\mu}$.

Let ${\rm N}_{\mu} \subset {\rm M}_{\mu}$ be the subspace of all the
$\emph{infinitesimally trivial}$ Beltrami differentials. Then the
tangent space of $T_{f}$ at $[\mu]$ is isomorphic to the quotient
space ${\rm M}_{\mu}/{\rm N}_{\mu}$.

Let $\mu$ be a Beltrami coefficient defined on $S^{2} \setminus
Q_{f}$. Let $\xi$ be a tangent vector of $T_{f}$ at $[\mu]$ which is
identified with a Beltrami differential $\xi(z)
\frac{d\overline{z}}{dz}$ defined on $\phi_{\mu}(S^{2} \setminus
Q_{f})$.

\begin{definition}\label{Tech-norm}{\rm
The Teichm\"{u}ller norm of the tangent vector $\xi $ is defined to
be
$$
\|\xi\| = \sup \: \Big |\int_{\phi_{\mu}(S^{2} \setminus Q_{f})}
q(z) \xi(z) dz \wedge d\overline{{z}}\Big|, $$ where the $\sup$ is
taken over all $q(z) dz^{2} \in {\rm  A}_{\mu}$ with $
\int_{\phi_{\mu}(S^{2} \setminus Q_{f})} |q(z)| dz \wedge
d\overline{{z}} = 1. $}
\end{definition}

\begin{definition}\label{Tech-metric}{\rm
Let $[\mu], [\nu] \in T_{f}$. The Teichm\"{u}ller distance
$d_{T}([\mu], [\nu])$ is define to be
$$
\frac{1}{2} \inf \log K(\phi_{\mu'} \circ \phi_{\nu'}^{-1})
$$
where $\phi_{\mu'}$ and $\phi_{\nu'}$ are quasi-conformal mappings
with Beltrami coefficients $\mu'$ and $\nu'$ and the inf is taken
over all $\mu'$ and $\nu'$ in the same Teichm\"{u}ller classes as
$\mu$ and $\nu$, respectively.}
\end{definition}

\begin{lemma}\label{infinitesimal-metric}
Let $\mu$ and $\nu$ be two Beltrami coefficients defined on $S^{2}
\setminus Q_{f}$. Then
$$
d_{T}([\mu], [\nu]) = \inf \int_{0}^{1} \|\tau'(t)\| dt
$$
where inf is taken over all the piecewise smooth curves $\tau(t)$ in
$T_{f}$ such that $\tau(0) = [\mu]$ and $\tau(1) = [\nu]$.
\end{lemma}
%%%%%%%%%%%%%%%%%%%%%%%%%%%%%%%%%%%%%%%%%%%%%%%%%%%%%%%%%%%%%%%%%%%%%
%%%%%%%%%%%%%%%%%%%%%%%%%%%%%%%%%%%%%%%%%%%%%%%%%%%%%%%%%%%%%%%%%%%%%
%%%%%%%%%%%%%%%%%%%%%%%%%%%%%%%%%%%%%%%%%%%%%%%%%%%%%%%%%%%%%%%%%%%%%

\section{The pull-back operator }

As in the post-critically finite case, we may assume that $f$ is a
quasiconformal map(This  is because except the finite holomorphic
disks, there are only finitely many points in $P_{f}$, and
therefore, the CLH-equivalent class of $f$ must contain a
quasiconformal branched covering map of the sphere). From now on, we
use ${\Bbb P}^{1}$ to denote to  the two sphere endowed with the
standard complex structure.

Remind that for a Beltrami coefficient $\mu$ defined defined on the
sphere $S^{2}$, the pull back of $\mu$ by $f$, which is denoted by
$f^{*}(\mu)$,  is defined to be
\begin{equation}\label{pull-back-coefficients}
(f^{*}\mu)(z) = \frac{\mu_f(z) + \mu(f(z)) \theta(z)}{1 +
\overline{\mu_f (z)} \mu(f(z)) \theta(z)}
\end{equation}
where $\theta(z)$ = $\overline{f_{z}}/f_{z}$ and $\mu_{f}(z) =
f_{\bar{z}} /f_{z}$. It is important to note that if $\mu$ depends
complex analytically on $t$, then so does $f^{*}(\mu)$.

Now let $\mu(z)$ be a Beltrami coefficient defined on $S^{2}
\setminus Q_{f}$. Define the Beltrami coefficient
$\text{Ext}(\mu)(z)$ on $S^{2}$ by setting
\begin{equation}
\text{Ext}(\mu)(z)  =
\begin{cases}
\mu(z) & \text{ for $z \in S^{2} \setminus Q_{f}$},\\
0 & \text{ for otherwise}.
\end{cases}
\end{equation}

By (\ref{pull-back-coefficients}),  $f^{*}(\text{Ext}(\mu))$  is a
Beltrami coefficient on the sphere $S^{2}$.  Let us simply use
$f^{*}(\mu)$ to denote the restriction of $f^{*}(\text{Ext}(\mu))$
on $S^{2} \setminus Q_{f}$.

\begin{lemma}\label{pull-back}
The map $f^{*}$ induces a complex analytic operator $\sigma_{f}:
T_{f} \to T_{f}$.
\end{lemma}
\begin{proof}
Suppose $\mu$ and $\nu$ are two Beltrami coefficients defined  on
$S^{2} \setminus Q_{f}$ which are equivalent to each other. Let
$\text{Ext}(\mu)$ and $\text{Ext}(\nu)$  be their extensions to
$S^{2}$. Let $\phi_{\text{Ext}(\mu)}$ and $\phi_{\text{Ext}(\nu)}$
be the corresponding quasiconformal homeomorphisms of the sphere
which fix $0$, $1$, and the infinity. Let $\phi_{\mu}$ and
$\phi_{\nu}$ denote their restrictions to $S^{2} \setminus Q_{f}$ ,
respectively.  Since $\mu$ is equivalent to $\nu$,  we have a
holomorphic isomorphism
$$
h: {\Bbb P}^{1}  \setminus \phi_{\text{Ext}(\nu)}(Q_{f}) \to  {\Bbb
P}^{1}  \setminus \phi_{\text{Ext}(\mu)}(Q_{f})
$$
such that $\phi_{\mu}$ is isotopic to $h\circ \phi_{\nu}$ rel
$X_{f}$. Now define a homeomorphism $\text{Ext}(h): {\Bbb P}^{1} \to
{\Bbb P}^{1}$ by setting

\begin{equation}
\text{Ext}(h)(z) =
\begin{cases}
h(z) & \text{ for $z\in {\mathbb P}^{1}\setminus
\phi_{\text{Ext}(\nu)}(Q_{f})$}, \\
\phi_{\text{Ext}(\mu)}\circ \phi_{\text{Ext}(\nu)}^{-1}(z) & \text{
for otherwise}.
\end{cases}
\end{equation}

It is clear that $\text{Ext}(h)$ is holomorphic everywhere except
those points in $\phi_{\text{Ext}(\nu)}(X_{f})$. Since
$\phi_{\text{Ext}(\nu)}(X_{f})$ is the union of finitely many points
and finitely many quasi-circles(see Remark~\ref{analytic-boundary}),
it follows that $\hbox{Ext}(h)$ is a holomorphic homeomorphism of
the sphere to itself, and therefore a M\"{o}bius map. By the
normalization condition, $\text{Ext}(h)$ fixes $0$, $1$, and
$\infty$ also. So $\text{Ext}(h) = \hbox{id}$. This implies that
$\phi_{\mu}$ and $\phi_{\nu}$ are isotopic to each other rel
$X_{f}$, and in particular, $\phi_{\mu} = \phi_{\nu}$ on $X_{f}$.
Since $\phi_{\text{Ext}(\mu)}$ and $\phi_{\text{Ext}(\nu)}$ are
holomorphic on $D_{f}$, it follows that $\phi_{\text{Ext}(\mu)} =
\phi_{\text{Ext}(\nu)}$ on $Q_{f}$ and therefore are isotopic to
each other rel $Q_{f}$. Since $f(Q_{f}) \subset Q_{f}$, we can
therefore lift this isotopy and get a isotopy between
$\phi_{f^{*}(\text{Ext}(\mu))}$ and $\phi_{f^{*}(\text{Ext}(\nu))}$
rel $Q_{f}$. It follows that $\phi_{f^{*}(\mu)}$ and
$\phi_{f^{*}(\nu)}$, which are respectively the restrictions of
$\phi_{f^{*}(\text{Ext}(\mu))}$ and $\phi_{f^{*}(\text{Ext}(\nu))}$
on $S^{2} \setminus Q_{f}$,  are isotopic to each other rel $X_{f}$.
This implies that $[f^{*}(\mu)]= [f^{*}(\nu)]$. Let
$\sigma_{f}([\mu]) = [f^{*}(\mu)]$.

Now let us show that $\sigma_{f}$ is complex analytic. Suppose that
we have a curve $\tau(t)$ in $T_{f}$ such that $\tau(t)$ depends
complex analytically on $t$ when $t$ varies in a small disk $\{t\:
\big{|}\: |t| < \epsilon\}$.  We may assume that $\epsilon > 0$ is
small enough so that the following arguments are valid.  Let $[\mu]
= \tau(0)$.  Then the map $\phi_{\mu}$ induces an isometry between
$T_{f}$ and the Teichm\"{u}ller space modeled on $({\Bbb P}^{1}
\setminus \phi_{\mu}(Q_{f}), \phi_{\mu}(X_{f}))$. This isometry maps
the curve $\tau(t)$ to a complex  analytic curve $\theta(t)$, $|t| <
\epsilon$, which passes  through the origin. Since $\epsilon
> 0$ is small, by Ahlfors-Weill's formula(see Lemma 7, Chapter 5 of
\cite{Ga}),  there is a curve of Beltrami coefficients $\eta(t)$
defined on ${\Bbb P}^{1} \setminus \phi_{\mu}(Q_{f})$  such that
$[\eta(t)] = \theta(t)$ and $\eta(t)$ depends complex analytically
on $t$ when $t$ varies in the disk $\{t\: \big{|}\: |t| <
\epsilon\}$.  Using formula (\ref{pull-back-coefficients}), we can
pull back  $\eta(t)$ by $\phi_{\mu}$ and get a curve of Beltrami
coefficients $\gamma(t)$, $|t| < \epsilon$,  defined on $S^{2}
\setminus Q_{f}$. It follows that  $[\gamma(t)] = \tau(t)$ and
$\gamma(t)$ depends complex analytically on $t$ when $t$ varies in
the disk $\{t\: \big{|}\: |t| < \epsilon\}$. From
(\ref{pull-back-coefficients}),  it follows that
$\tilde{\gamma}(t)$, $|t| < \epsilon$,  is also a curve of complex
analytic Beltrami coefficients defined on $S^{2} \setminus Q_{f}$.
Now by Bers Embedding Theorem(see Theorem 1, Chapter 5 of
\cite{Ga}),  the curve $\sigma_{f}(\tau(t)) = [\tilde{\gamma}(t)]$
is a curve in $T_{f}$ which depends complex analytically on $t$ when
$t$ varies in the disk $\{t\: \big{|}\: |t| < \epsilon\}$. This
proves that $\sigma_{f}$ is a complex analytic operator.

\end{proof}

Once no confusion is caused,  let us simply use $\mu$ to denote
either $\text{Ext}(\mu)$ or $\mu$. Let $\tilde{\mu}(z) =
f^{*}(\mu)$.

Let $\phi_{\mu}, \phi_{\tilde{\mu}}: S^{2} \to {\Bbb P}^{1}$ denote
the quasiconformal homeomorphisms which fix $0$, $1$, and the
infinity and which solve the Beltrami equations given by $\mu$ and
$\tilde{\mu}$, respectively. Let  $$g = \phi_{\mu} \circ f \circ
\phi_{\tilde{\mu}}^{-1}.$$ It is clear that  $g$ is a rational map
and the following diagram commutes.
$$
\begin{array}{ccc}
      (S^{2}, Q_f)& {\buildrel
\phi_{\tilde{\mu}} \over \longrightarrow} &({\Bbb P}^{1}, \phi_{\tilde{\mu}}(Q_{f}))\\
            \downarrow f &                                            &\downarrow
            g\\
            (S^{2}, Q_{f})& {\buildrel \phi_{\mu} \over \longrightarrow} & ({\Bbb
           P}^{1},\phi_{\mu}(Q_{f}))
\end{array}
$$

Now suppose that $\xi$ is a tangent vector of $T_{f}$ at $\tau =
[\mu]$. This means that there is a smooth curve of Beltrami
coefficients $\gamma(t)$ defined on $S^{2} \setminus Q_{f}$, such
that $\gamma(0) = \mu$ and
\begin{equation}\label{tangent-vector}
\xi = \frac{d}{dt}\bigg{|}_{t=0}\mu_{\phi_{\gamma(t)} \circ
\phi_{\mu}^{-1}}
\end{equation}

Let $d \sigma_{f}\big{|}_{\tau}$ denote the tangent map of
$\sigma_{f}$ at $\tau$. Let $\tilde{\xi} = d
\sigma_{f}\big{|}_{\tau} (\xi)$.
\begin{lemma}\label{derivaltive-formula}
Let $\xi$ and $\tilde{\xi}$ be as above. Then
\begin{equation}\label{derivative}
\tilde{\xi}(w) = \xi(g(w)) \frac{\overline{g'(w)}}{g'(w)}.
\end{equation}
\end{lemma}
\begin{proof}
Note that
$$
\tilde{\xi} = \frac{d}{dt}\bigg{|}_{t=0}\mu_{\phi_{\gamma(t)} \circ
f\circ  \phi_{\tilde{\mu}}^{-1}} =
\frac{d}{dt}\bigg{|}_{t=0}\mu_{\phi_{\gamma(t)}\circ \phi_{\mu}^{-1}
\circ \phi_{\mu} \circ f\circ \phi_{\tilde{\mu}}^{-1}} =
\frac{d}{dt}\bigg{|}_{t=0}\mu_{\phi_{\gamma(t)}\circ \phi_{\mu}^{-1}
\circ g}.
$$
Since $g$ is a rational map, by (\ref{pull-back-coefficients}) we
have
$$
\mu_{\phi_{\gamma(t)}\circ \phi_{\mu}^{-1} \circ g} (w) =
\mu_{\phi_{\gamma(t)}\circ \phi_{\mu}^{-1}}(g(w))
\frac{\overline{g'(w)}}{g'(w)}
$$
The Lemma  then follows from (\ref{tangent-vector}).
\end{proof}

Let $\tilde{q}=\tilde{q}(w) dw^2$ be a non-zero integrable
holomorphic quadratic differential defined on ${\mathbb
P}^{1}\setminus \phi_{\tilde{\mu}}(Q_f)$.  Define
\begin{equation}\label{push-forward}
q(z)  = \sum_{g(w) = z} \frac{\tilde{q}(w)}{[g'(w)]^{2}}.
\end{equation}
It is easy to see that $q = q(z)dz^{2}$ is a holomorphic quadratic
differential defined on ${\mathbb P}^{1}\setminus \phi_{\mu}(Q_f)$.
\begin{proposition}\label{key}
$$ \int_{{\mathbb P}^{1}\setminus \phi_{\mu}(Q_f)}|q(z)|dz \wedge
d\overline{{z}} \le  \int_{{\mathbb P}^{1}\setminus
\phi_{\tilde{\mu}}(Q_f)}| \tilde{q}(w)|dw \wedge d\overline{w}
-\int_{\cup_{i} \phi_{\tilde{\mu}}(A_{i})}| \tilde{q} (w)|dw \wedge
d\overline{w}.
$$
\end{proposition}

\begin{proof}
$$
\int_{{\mathbb P}^{1}\setminus \phi_{\mu}(Q_f)} |q(z)| dz \wedge
d\overline{z}=\int_{{\mathbb P}^{1}\setminus \phi_{\mu}(Q_f)} \Big|
\sum_{g(w)=z} \frac{\tilde{q}(w)}{[g' (w)]^{2}}\Big| dz \wedge
d\overline{z}
$$
$$
 \leq
\int_{\big{(}{\mathbb P}^{1}\setminus \phi_{\tilde{\mu}}(Q_f)\big{)}
\setminus \big( \cup_{i} \phi_{\tilde{\mu}}(A_{i})\big)}
|\tilde{q}(w)| dw \wedge d\overline{w}
$$
$$
=\int_{{\mathbb P}^{1}\setminus \phi_{\tilde{\mu}}(Q_f)}|
\tilde{q}(w)|dw \wedge d\overline{w} -\int_{\cup_{i}
\phi_{\tilde{\mu}}(A_{i})}| \tilde{q} (w)|dw \wedge d\overline{w}
$$
The first inequality comes from the fact $f(\cup A_{i}) \subset \cup
D_{i}$.   This completes the proof of Proposition~\ref{key}.
\end{proof}
\begin{proposition}\label{key'} $\int_{{\Bbb P}^{1} \setminus
\phi_{\tilde{\mu}}(Q_{f})} \tilde{\xi} (w) \tilde{q}(w)dw \wedge
d\overline{{w}} =\int_{{\Bbb P}^{1} \setminus  \phi_{\mu}(Q_{f})}
\xi (z) q(z)dz \wedge d\overline{{z}} $.
\end{proposition}
\begin{proof}

Note that $\phi_{\tilde{\mu}}(Q_{f}) \subset
g^{-1}(\phi_{\mu}(Q_{f}))$ and by (\ref{derivative}) $\tilde{\xi}(w)
= 0$ for all $w \in g^{-1}(\phi_{\mu}(Q_{f})) \setminus
\phi_{\tilde{\mu}}(Q_{f})$.  We thus have
$$
 \int_{{\Bbb P}^{1} \setminus \phi_{\tilde{\mu}}(Q_{f})}
\tilde{\xi} (w) \tilde{q}(w)dw \wedge d\overline{{w}} =\int_{{\Bbb
P}^{1} \setminus g^{-1}(\phi_{\mu}(Q_{f}))} \tilde{\xi} (w)
\tilde{q}(w)dw \wedge d\overline{{w}}.
$$

Now Proposition~\ref{key'} follows from (\ref{derivative}),
(\ref{push-forward}), and the fact that
$$
dw \wedge d\overline{{w}} = \frac{dz \wedge
d\overline{{z}}}{|g'(w)|^{2}}.
$$

\end{proof}

As a direct consequence of Propositions \ref{key} and \ref{key'}, we
have
\begin{corollary}\label{key-6}{\rm Let $\tau \in T_{f}$. Then
$\| d \sigma_{f} \big{|}_{\tau} \| \le 1$. }
\end{corollary}
\begin{remark}\label{Ka}{\rm
Corollary~\ref{key-6} also follows from the general fact that a
complex analytic operator does not increase the Kabayashi's metric.
But this particular argument we used here  will be established in
the latter sections to prove a strict inequality(see
Corollary~\ref{key-3}). }
\end{remark}
The next lemma reduces the proof of the Main Theorem to showing that
the pull back operator $\sigma_{f}$ has a unique fixed point in
$T_{f}$.

\begin{lemma}~\label{fp}
The map $f$ is CLH-equivalent to a unique rational map (up to
M\"{o}bius conjugations) if and only if $\sigma_{f}$ has a unique
fixed point in $T_{f}$.
\end{lemma}

\begin{proof}
If $\sigma_{f}$ has a fixed point $[\mu]$ in $T_{f}$, then
$\tilde{\mu} = f^{*}\mu\sim \mu$. Let $\hbox{Ext}(\mu)$ be the
extension of $\mu$ to $S^{2}$.  Let $\phi_{\text{Ext}(\mu)}$ and
$\phi_{f^{*}(\text{Ext}(\mu))}$ be the corresponding quasiconformal
homeomorphisms which fix $0$, $1$, and the infinity. Let
$\phi_{\mu}$ and $\phi_{\tilde{\mu}}$ be their restrictions to
$S^{2} \setminus Q_{f}$, respectively.  It follows that there is a
conformal isomorphism
$$
h: {\Bbb P}^{1}\setminus \phi_{\mu}( Q_{f}) \to {\Bbb
P}^{1}\setminus \phi_{\tilde{\mu}}( Q_{f})
$$
such that $\phi_{\tilde{\mu}}$ and $h\circ \phi_{\mu}$ are isotopic
to each other rel $X_{f}$. As in the proof of Lemma~\ref{pull-back},
one can show that such $h$ is actually equal to the identity map. In
fact,  we can again define a homeomorphism $\text{Ext}(h): {\Bbb
P}^{1} \to {\Bbb P}^{1}$ by setting

\begin{equation}
\text{Ext}(h)(z) =
\begin{cases}
h(z) & \text{ for $z\in {\mathbb P}^{1}\setminus
\phi_{\text{Ext}(\mu)}(Q_{f})$}, \\
\phi_{f^{*}(\text{Ext}(\mu))}\circ \phi_{\text{Ext}(\mu)}^{-1}(z) &
\text{ for otherwise}.
\end{cases}
\end{equation}
It is clear that $\text{Ext}(h)$ is holomorphic everywhere except
those points in $\phi_{\text{Ext}(\mu)}(X_{f})$. Since
$\phi_{\text{Ext}(\mu)}(X_{f})$ is the union of finitely many points
and finitely many quasi-circles(see Remark~\ref{analytic-boundary}),
it follows that $\hbox{Ext}(h)$ is a holomorphic homeomorphism of
the sphere to itself, and therefore a M\"{o}bius map. By the
normalization condition, $\text{Ext}(h)$ fixes $0$, $1$, and
$\infty$ also. So $\text{Ext}(h) = \hbox{id}$. This  implies that
$\phi_{\mu}$ and $\phi_{\tilde{\mu}}$ are isotopic to each other rel
$X_{f}$. It follows  that $\phi_{\text{Ext}(\mu)}$ and
$\phi_{f^{*}(\text{Ext}(\mu))}$ are isotopic to each other rel
$Q_{f}$. Note that when restricted to $D_{f}$,
$\phi_{\text{Ext}(\mu)}$ and $\phi_{f^{*}(\text{Ext}(\mu))}$ are
analytic and equal to each other. This implies that $f$ is
CLH-equivalent to the rational map $g= \phi_{\text{Ext}(\mu)}\circ
f\circ \phi_{f^{*}(\text{Ext}(\mu))}^{-1}$.

If $f$ is CLH-equivalent to $g$, then we have a Beltrami coefficient
$\mu$ defined on $S^{2} \setminus  Q_{f}$ such that $g=
\phi_{\text{Ext}(\mu)}\circ f\circ
\phi_{f^{*}(\text{Ext}(\mu))}^{-1}$ and moreover,
$\phi_{\text{Ext}(\mu)}$ and $\phi_{f^{*}(\text{Ext}(\mu))}$ are
isotopic to each other rel $Q_{f}$. This implies that $\phi_{\mu}$
and $\phi_{\tilde{\mu}}$ are isotopic to each other rel $X_{f}$. It
follows that  $[f^{*}(\mu)] = [\mu]$ and thus $\sigma_{f}([\mu]) =
[\mu]$.

It is clear that the fixed point $[\mu]$ is unique is equivalent to
say that $g$ is unique up to M\"{o}bius conjugations.
\end{proof}

%%%%%%%%%%%%%%%%%%%%%%%%%%%%%%%%%%%%%%%%%%%%%%%%%%%%%%%%%%%%%%%%%%%%%
%%%%%%%%%%%%%%%%%%%%%%%%%%%%%%%%%%%%%%%%%%%%%%%%%%%%%%%%%%%%%%%%%%%%%
%%%%%%%%%%%%%%%%%%%%%%%%%%%%%%%%%%%%%%%%%%%%%%%%%%%%%%%%%%%%%%%%%%%%%

\section{Bounded geometry}

Let $d(X, Y)$ denote  the spherical distance between two subsets of
the sphere. Recall that
$$
D_{f} =\cup_{i} D_{i}\:,\:P_{1} = P_f\setminus D_{f}, \: \hbox{and}
\: P_{f}' = \{a_{i}\}.
$$

\begin{definition}~\label{bounded geometry}{\rm
Let $b>0$ be a constant. Let $T_{f, b} \subset T_{f}$  be the
subspace such that for every $[\mu] \in T_{f, b}$, the following
conditions hold,
\begin{enumerate}
\item for all $z_{i}\not=z_{i'} \in P_{1}$,
$$
d(\phi_{\mu}(z_{i}), \phi_{\mu}(z_{i'}))\geq b;
$$
\item for all $z_{j}\in P_{1}$ and all $D_{i}$,
$$
d(\phi_{\mu}(z_{j}), \phi_{\mu}(D_{i}))\geq b;
$$
\item for all $D_{i}\not= D_{i'}$,
$$
d(\phi_{\mu}(D_{i}), \phi_{\mu}(D_{i'}))\geq b;
$$
\item for every $D_{i}$, $\phi_{\mu}(D_{i})$
contains a round disk of radius $b$ centered at $\phi_{\mu}(a_{i})$,
\end{enumerate}
where $\phi_{\mu}: S^{2} \to {\Bbb P}^{1}$ is the quasiconformal
homeomorphism which fixes $0$, $1$, and the infinity, and which
solves the Beltrami equation given by $\hbox{Ext}(\mu)$.}
\end{definition}

Let $K > 1$. Then   the family of all the $K-$quasiconformal
homeomorphisms of the sphere to itself, which fix $0$, $1$, and the
infinity, is compact. We thus have
\begin{lemma}\label{tech-2}
Let $K > 1$. Then for every $\delta > 0$, there is an $\epsilon > 0$
depending only on $K$ and $\delta$ such that for every two points
$x, y \in {\Bbb P}^{1}$ with $d(x, y)
> \delta$, and every
$K-$quasiconformal homeomorphism $\phi: {\Bbb P}^{1} \to {\Bbb
P}^{1}$ which fixes $0$, $1$, and the infinity, we have $d(\phi(x),
\phi(y)) > \epsilon$.
\end{lemma}

By Definitions~\ref{Tech-metric} and \ref{bounded geometry}, and
Lemma~\ref{tech-2}, we have
\begin{lemma}\label{tech-6}
Let $b, D > 0$. Then there is a $b'> 0$ depending only on $b$ and
$D$ such that for any two Beltrami coefficients $\mu$ and $\nu$
defined on $S^{2} \setminus Q_{f}$, if $d_{T}([\mu], [\nu]) < D$ and
$\mu \in T_{f,b}$, then $\nu \in T_{f, b'}$.
\end{lemma}

\begin{definition}\label{norm-of -curve}{\rm
Let $Z$ be a subset of $S^{2}$ with $\#(Z) \ge 4$.  Let $[\mu]\in
T_{f}$ and $\gamma \subset S^{2} \setminus Z$ be a simple closed and
non-peripheral curve.  We use $\|\gamma\|_{\mu, Z}$ to denote the
hyperbolic length  of the unique simple closed geodesic $\xi$ which
is homotopic to $\phi_{\mu}(\gamma)$ in the hyperbolic Riemann
surface $\Bbb{P}^{1} \setminus \phi_{\mu}(Z)$.   We say $\gamma$ is
a $(\mu, Z)$-simple closed geodesic if $\phi_{\mu}(\gamma)$ is a
simple closed geodesic in $\Bbb{P}^{1} \setminus \phi_{\mu}(Z)$.}
\end{definition}

For each holomorphic disk $D_{i}$, fix a point $b_{i}$ on the
boundary $\partial D_{i}$. Set
$$
E = P_{1}\cup \cup_{i}\{a_{i}, b_{i}\}.
$$

Note that  $P_{1}$ contains $0$, $1$, and the infinity by
assumption. Since $P_{1} \subset E$ and $\phi_{\mu}$ fixes $0$, $1$,
and the infinity, it follows that $E$ and $\phi_{\mu}(E)$ contain
$0$, $1$, and the infinity also.

\begin{lemma}~\label{gdb1}
Let $a > 0$. Then there is a  $b>0$ depending only on $a$ such that
for every Beltrami coefficient $\mu$ defined on $S^{2} \setminus
Q_{f}$ with $\mu(z) =  0$ on $\cup_{i} A_{i}$, if every $(\mu,
E)$-simple closed geodesic $\gamma \subset S^{2} \setminus Q_{f}$
has hyperbolic length not less than $a$, then $\mu\in T_{f,b}$.
\end{lemma}

\begin{proof}
Note that $\#(\phi_{\mu}(E)) = \#(E)$ is finite.  Since
$\phi_{\mu}(E)$ contains $0$, $1$, and the infinity,    it follows
that the spherical distance between any two points in
$\phi_{\mu}(E)$ has a positive lower bound which depends only on $a$
and $\#(E)$.   Since $\phi_{\mu}$ is holomorphic in every
topological disk $\overline{D_{i}} \cup A_{i}$ and since
$\phi_{\mu}(\overline{D}_{i})$ contains $\phi_{\mu}(a_{i})$ and
$\phi_{\mu}(b_{i})$,   it  follows from Koebe's distortion theorem
that  every  $\phi_{\mu}(D_{i})$ contains a round disk centered at
$\phi_{\mu}(a_{i})$, the radius of which has a positive lower bound
depending only on $a$.  Since $\{0, 1, \infty\} \notin
\phi_{\mu}(\overline{D_{i}}\cup A_{i})$, it follows that the
diameter of each component of ${\Bbb P}^{1} \setminus
\phi_{\mu}(A_{i})$ has a positive lower bound depending only on $a$.
Since $\phi_{\mu}$ is analytic on every $A_{i}$, we have
$$
\hbox{mod}(\phi_{\mu}(A_{i}))
=\hbox{mod}(A_{i}).
$$
It follows that every $\phi_{\mu}(A_{i})$ has  definite thickness
which depends only on $a$. All of these implies that there is a
constant $b > 0$ depending only on $a$ such that  the four
conditions in Definiton~\ref{bounded geometry} hold. The proof of
the lemma is completed.
\end{proof}
The next lemma is a direct consequence of Proposition 6.1 and
Theorem 6.3 of \cite{DH}.
\begin{lemma}\label{basic-1}
 Let $X$ be a hyperbolic Riemann surface and $\gamma \subset X$ be a
 simple closed geodesic with hyperbolic length $l$. Then there
 exists a topological annulus $A \subset X$ such that
 \begin{itemize}
 \item[1.] $\gamma$ is  the
 core curve of $A$,
 \item[2.] $\frac{\pi}{2l} - 1 < {\rm mod}(A) < \frac{\pi}{2l}$.
 \end{itemize}
 \end{lemma}

From the modulus inequality of Teichm\"{u}ller extremal problem(For
instance,  see Chapter III of \cite{A}), we have
\begin{lemma}\label{basic-2}
Let  $T \in {\Bbb P}^{1} \setminus \{0, 1, \infty\}$. Let $H \subset
{\Bbb P}^{1}$ be an annulus   which separates $\{0, 1\}$ and $\{T,
\infty\}$.  Then
$$
 {\rm mod}(H) \le
\frac{1}{2 \pi } \log 16(|T| +1).
$$
\end{lemma}

\begin{lemma}~\label{collar}
There exists an $\eta > 0$  such that for any Beltrami coefficient
$\mu$ defined on $S^{2}\setminus Q_{f}$ with $\mu(z) = 0$ on
$\cup_{i} A_{i}$ and any $(\mu, E)$-simple closed geodesic $\gamma
\subset S^{2}\setminus E$ with $\|\gamma\|_{\mu, E} < \eta$, we have
$\gamma \subset  S^{2} \setminus Q_{f}$. Moreover, for any $\epsilon
> 0$, there is a $\delta > 0$ such that
\begin{equation}\label{geodesic-comp}
\|\gamma\|_{\mu, E} > (1 - \epsilon) \|\gamma\|_{\mu, Q_{f}}
\end{equation}
provided that $\|\gamma\|_{\mu, E} < \delta$.
\end{lemma}
\begin{proof}
Let $\gamma \subset S^{2}\setminus E$ be a $(\mu, E)$-simple closed
geodesic. By Lemma~\ref{basic-1}, there is an annulus $A \subset
{\Bbb P}^{1}
 \setminus \phi_{\mu}(E)$ such that $\phi_{\mu}(\gamma)$ is the core curve of $A$
 and
\begin{equation}\label{ine-1}
\frac{\pi}{2\|\gamma\|_{\mu, E}} - 1 < {\rm mod}(A) <
\frac{\pi}{2\|\gamma\|_{\mu, E}}.
\end{equation}
 We may assume that $A$ separates $0$ and the infinity.
 Let $K_{1}$ and $K_{2}$ be the two components of ${\Bbb P}^{1}
 \setminus A$ such that $0 \in K_{1}$ and $\infty \in K_{2}$.  Let
 $$
 r = \max\{|z|\: \big{|}\: z \in  K_{1}\} \quad
 \hbox{and} \quad  R = \min\{ |z|\:\big{|}\: z \in
 K_{2}\}.
 $$
 By Lemma~\ref{basic-2}, when $\|\gamma\|_{\mu, E}$ is small, $R/r$
 is large. Consider the round annulus
 $$
 H = \{z\:\big{|}\: r < |z| < R\}.
 $$
It follows that  $H \subset A$ and that the core curve of $H$ is in
the same homotopic class as $\gamma$.  By  Lemma~\ref{basic-2} and
(\ref{ine-1}), it follows that there is a uniform constant $0< C <
\infty$ such that
\begin{equation}\label{mod-ine-2}
{\rm  mod}(H) \ge {\rm mod}(A) - C
\end{equation}
holds provided that $\|\gamma\|_{\mu, E}$ is small.  Note that every
pair $\{\phi_{\mu}(a_{i}), \phi_{\mu}(b_{i})\}$ is contained either
in $\{z\:\big{|}\: |z| < r\}$ or in $\{z\:\big{|}\: |z| > R\}$.
Since $\phi_{\mu}$
 is holomorphic in $\overline{D_{i}} \cup A_{i}$ and $\{\phi_{\mu}(a_{i}),
 \phi_{\mu}(b_{i})\} \subset \overline{D_{i}}$, it follows from Koebe's distortion theorem
that there is an $1 < M < \infty$, which depends only on $\{D_{i}\}$
and $\{A_{i}\}$,  such that every $\phi_{\mu}(\overline{D_{i}})$ is
contained either in $\{z\:\big{|}\: |z| < Mr\}$ or in
$\{z\:\big{|}\: |z| > R/M\}$. By (~\ref{ine-1}) and
(\ref{mod-ine-2}), we have $$ R/M
> Mr$$ provided that $\|\gamma\|_{\mu, E}$ is small enough. All of these implies  that the annulus
$$
H_{M} = \{z\:\big{|}\: Mr < |z| < R/M\}
$$
is contained in ${\Bbb P}^{1} \setminus \phi_{\mu}(Q_{f})$ provided
that $\|\gamma\|_{\mu, E}$ is small enough.

Now the first assertion of the lemma follows if we can show that
$$\phi_{\mu}(\gamma) \subset H_{M}$$ provided that $\|\gamma\|_{\mu,
E}$ is small enough.  Suppose this were not true. Then there are two
cases. In the first case, there exist two points $z$ and $z'$ such
that
\begin{itemize}
\item[1.] $z \in K_{2}$ with $|z| = R$,
\item[2.] $|z'| = R/M$,
\item[3.] $\phi_{\mu}(\gamma)$ separates $\{0, z'\}$ and $\{z, \infty\}$.
\end{itemize}
In the second case,  there exist two points $z$ and $z'$ such that
\begin{itemize}
\item[1.] $|z| = Mr$,
\item[2.] $z' \in K_{1}$ and $|z'| = r$.
\item[3.] $\phi_{\mu}(\gamma)$ separates $\{0, z'\}$ and $\{z, \infty\}$.
\end{itemize}
Suppose we are in the first case.   Note that the curve
$\phi_{\mu}(\gamma)$ separates $A$ into two sub-annuli such that the
modulus of each of them is equal to ${\rm mod}(A)/2$. But on the
other hand,  the outer one separates $\{0, z'\}$ and $\{z,
\infty]\}$, and thus by Lemma~\ref{basic-2}, its modulus has an
upper bound depending only on $M$. By (\ref{ine-1}) this is
impossible when $\|\gamma\|_{\mu, E}$ is small enough. The same
argument can be used to get a contradiction in the second case. This
proves the first assertion of the Lemma.

Now  let us prove the second assertion.  Let $l$ denote the
hyperbolic length of the core curve of $H_{M}$ with respect to the
hyperbolic metric of $H_{M}$. Since $H_{M} \subset {\Bbb P}^{1}
\setminus \phi_{\mu}(Q_{f})$ when $\|\gamma\|_{\mu, E}$ is small
enough, it follows that $l > \|\gamma\|_{\mu, Q_{f}}$. Thus we have
$$
{\rm mod}(H_{M}) = \frac{\pi}{2l}  < \frac{\pi}{2 \|\gamma\|_{\mu,
Q_{f}}}.
$$
From  (~\ref{ine-1}) and (\ref{mod-ine-2}), there is a constant $0<
C' < \infty$ such that
$${\rm mod}(H_{M}) \ge \frac{\pi}{2 \|\gamma\|_{\mu, E}} - C'
$$
holds provided that $\|\gamma\|_{\mu, E}$ is small enough.   Thus we
have
$$
\frac{\pi}{2 \|\gamma\|_{\mu,Q_{f}}}\le \frac{\pi}{2
\|\gamma\|_{\mu, E}} \le \frac{\pi}{2 \|\gamma\|_{\mu,Q_{f}}} + C'.
$$
 The second assertion follows.
\end{proof}

%%%%%%%%%%%%%%%%%%%%%%%%%%%%%%%%%%%%%%%%%%%%%%%%%%%%%%%%%%%%%%%%%%%%%
%%%%%%%%%%%%%%%%%%%%%%%%%%%%%%%%%%%%%%%%%%%%%%%%%%%%%%%%%%%%%%%%%%%%%
%%%%%%%%%%%%%%%%%%%%%%%%%%%%%%%%%%%%%%%%%%%%%%%%%%%%%%%%%%%%%%%%%%%%%
%%%%%%%%%%%%%%%%%%%%%%%%%%%%%%%%%%%%%%%%%%%%%%%%%%%%%%%%%%%%%%%%%%%%%
%%%%%%%%%%%%%%%%%%%%%%%%%%%%%%%%%%%%%%%%%%%%%%%%%%%%%%%%%%%%%%%%%%%%%

\section{From Bounded geometry to strictly contracting}

The main purpose of this section is to prove that bounded geometry
implies the strict contracting property of the operator $\sigma_{f}:
T_{f} \to T_{f}$.  Let us first prove a technical lemma.
\begin{lemma}\label{tech}
Let $H = \{z\:\big{|}\: 1 < |z| < R\}$ be an annulus. Let $F_{n}(w)$
be a sequence of integrable and holomorphic functions defined on $H$
such that
\begin{equation}\label{assum}
\int_{H}|F_{n}(w)| dw \wedge d\overline{w} \to 0 \hbox{ as } n \to
\infty.
\end{equation}
Then  for any $1 < r < R$,
$$
\int_{|w| = r} |F_{n}(w)| |dw| \to 0 \hbox{ as } n \to \infty.
$$
\end{lemma}
\begin{proof}
Let $1 < r < R$ be fixed. Take $\delta > 0$ such that $1 + \delta <
r <R- \delta$.  Let
$$
C(r, \delta) = \min\{r- 1- \delta, R-\delta - r\}.
$$
It follows that $C(r, \delta) > 0$. For any $\epsilon > 0$,  by
(\ref{assum}), there is an $N$ such that for every $n > N$, there
exist  $1< R_{1} < 1 + \delta$ and $R-\delta  < R_{2} < R$, such
that
$$
\int_{|z| = R_{1}} |F_{n}(z)| |dz| < \epsilon
$$
and
$$
\int_{|z| = R_{2}} |F_{n}(z)| |dz| < \epsilon.
$$
For $|w| = r$, by Cauchy formula, we have
$$
|F_{n}(w)|  \le  \bigg{|}\frac{1}{2\pi i}\int_{{\Bbb T}_{R_{1}} \cup
{\Bbb T}_{R_{2}}} \frac{F_{n}(z)}{z - w} dz \bigg{|}.
$$
Note that $|z - w|  \ge C(r, \delta)$ for $|w| = r \hbox{ and } z\in
{\Bbb T}_{R_{1}} \cup {\Bbb T}_{R_{2}}$. This implies that
$$
|F_{n}(w)| \le  \frac{\epsilon}{\pi C(r, \delta)}
$$
holds for all $|w| = r$ and $n > N$. It follows that for all $n >
N$,
$$
\int_{|w| = r} |F_{n}(w)| |dw| \le  \frac{2 r \epsilon }{C(r,
\delta)}.
$$
The Lemma follows.
\end{proof}

For a Beltrami coefficient $\mu$ defined on $S^{2} \setminus Q_{f}$,
we use $\tilde{\mu}$ to denote $f^{*}(\mu)$.

\begin{lemma}~\label{expcont2}
Let $b>0$.  Then there is a constant $0< a < 1$ depending only on
$b$ such that if both $[\mu]$ and  $[\tilde{\mu}]$ belong to  $T_{f,
b}$, then
$$
\int_{\cup \phi_{\tilde{\mu}}(A_{i})} |\tilde{q} (w)| dw \wedge
d\overline{{w}} \geq a
$$
where $\tilde{q} (w) dw^{2}$ is any integrable holomorphic quadratic
differential defined on ${\Bbb P}^{1} \setminus
\phi_{\tilde{\mu}}(Q_{f})$ with
$$
\int_{{\Bbb P}^{1} \setminus \phi_{\tilde{\mu}}(Q_{f})} |\tilde{q}
(w)| dw \wedge d\overline{{w}} = 1.
$$
\end{lemma}

\begin{proof} Let us prove it by contradiction.
By using a M\"{o}bius transformation which fixes $0$ and $1$, and
maps  $\phi_{\tilde{\mu}}(a_{1})$ to the infinity, we may assume
that $\infty \in D_{1}$.  Since $\tilde{\mu} \in T_{f,b}$, such
M\"{o}bius transformation lies in a compact family and therefore the
assumption does not affect the validity of the proof.

Now let us suppose that there exist a sequence of pairs
$(\tilde{\mu}_{n}, \mu_{n})$ in $T_{f, b}$ and a sequence of
holomorphic quadratic differentials $\tilde{q}_{n}$ over ${\Bbb
P}^{1} \setminus \phi_{\tilde{\mu_{n}}}(Q_{f})$ such that
\begin{equation}\label{norm-1}
\int_{{\Bbb P}^{1} \setminus \phi_{\tilde{\mu}_{n}}(Q_{f})}
|\tilde{q}_{n} (w)| dw \wedge d\overline{{w}} = 1,
\end{equation}
and
\begin{equation}\label{con-1}
\int_{\cup \phi_{\tilde{\mu}_{n}}(A_{i})} |\tilde{q}_{n}(w)| dw
\wedge d\overline{w} \to 0  \quad \hbox{as}\quad n\to \infty.
\end{equation}

By Lemma~\ref{srl} $f(\cup_{i} \overline{A}_{i})\subset
\cup_{i}D_{i}$ and $f$ is holomorphic in $\overline{D_{i}} \cup
A_{i}$.  This, together with the fact that $\phi_{\mu_{n}}$ is
holomorphic on $\cup_{i}D_{i}$, implies that
$\phi_{\tilde{\mu}_{n}}$ is holomorphic and thus univalent on
$\cup_{i} (\overline{D}_{i} \cup A_{i}) $.

Note that  every ring $A_{i}$ is holomorphically isomorphic to some
annulus
$$
H_{i} = \{z\:\big{|}\: 1 < |z| < R_{i}\}.
$$
Let $\Phi_{i}: H_{i} \to A_{i}$ be a holomorphic isomorphism and let
${\Bbb T}_{r}$ denote the circle $\{z\:\big{|}\: |z| = r\}$. We
claim that for every $1 < r < R_{i}$,
\begin{equation}\label{claim-1}
\int_{\phi_{\tilde{\mu}_{n}}(\Phi_{i}({\Bbb T}_{r}))}
|\tilde{q}_{n}(w)| |dw| \to 0  \quad \hbox{as}\quad n\to \infty.
\end{equation}
In fact, from (\ref{con-1}), we have
\begin{equation}\label{imedia-2}
\int_{H_{i}} |\tilde{q}_{n}( (\phi_{\tilde{\mu}_{n}} \circ
\Phi_{i})(z))||(\phi_{\tilde{\mu}_{n}} \circ \Phi_{i})'(z)|^{2}dz
\wedge d\overline{z} \to 0 \quad \hbox{as}\quad n\to \infty.
\end{equation}
By Lemma~\ref{tech}, we have
$$
\int_{{\Bbb T}_{r}} |\tilde{q}_{n}( (\phi_{\tilde{\mu}_{n}} \circ
\Phi_{i})(z))||(\phi_{\tilde{\mu}_{n}} \circ \Phi_{i})'(z)|^{2} |dz|
\to 0 \quad \hbox{as}\quad n\to \infty.
$$
 Since $\phi_{\tilde{\mu}_{n}} \circ
\Phi_{i}$ is univalent on $H_{i}$, it follows from Koebe's
$1/4$-theorem that for every $1 < r < R_{i}$, there is a $C > 1$
depending only on $r, R_{i}$, and $b$ such that
\begin{equation}\label{imedia-1}
1/C <  |(\phi_{\tilde{\mu}_{n}} \circ \Phi_{i})'(z)| < C
\end{equation}
holds for all $z \in {\Bbb T}_{r}$. We thus have
$$
\int_{{\Bbb T}_{r}} |\tilde{q}_{n}( (\phi_{\tilde{\mu}_{n}} \circ
\Phi_{i})(z))||(\phi_{\tilde{\mu}_{n}} \circ \Phi_{i})'(z)| |dz| \to
0 \quad \hbox{as}\quad n\to \infty.
$$
This implies (\ref{claim-1}) and the claim has been proved. Now for
every $A_{i}$, take an arbitrary $1< r_{i} < R_{i}$ and let
\begin{equation}\label{curve-1}
\gamma_{i, n} = (\phi_{\tilde{\mu}_{n}} \circ \Phi_{i})({\Bbb
T}_{r_{i}}).
\end{equation}
For every $n$, Let  $R_{n}$  denote the component of ${\Bbb P}^{1}
\setminus \cup_{i} \gamma_{i,n}$ such that
$$
\partial R_{n} = \cup_{i} \gamma_{i,n}.
$$

Recall that $P_{1} = \{z_{j}\}$ and $P_{f}' = \{a_{i}\}$ are both
finite sets and each $\tilde{q}_{n}=\tilde{q}_{n}(w)dw^2$ has at
most simple poles at the points in
$\{\phi_{\tilde{\mu}_{n}}(z_{j})\}$. This implies that one can write
\begin{equation}\label{decomp}
\tilde{q}_{n}(w) =  \sum_{j}
\frac{b_{j,n}}{w-\phi_{\tilde{\mu}_{n}}(z_{j})} + g_{n}(w)
\end{equation}
where $g_{n}(w)$ is a holomorphic function  on ${\Bbb P}^{1}
\setminus \phi_{\tilde{\mu}_{n}}(\overline{D_{f}})$.

Since $\tilde{\mu}_{n} \in T_{f, b}$, it follows by taking a
subsequence if necessary, that we can assume that for every $a_{i}$,
the sequence
$$
a_{i,n}= \phi_{\tilde{\mu}_{n}} (a_{i})
$$
converges to a point $e_{i}$ with respect to  the spherical distance
as $n$ goes to $\infty$.  Since $\phi_{\tilde{\mu}_{n}}$ is
holomorphic in $\overline{D_{i}} \cup A_{i}$, similarly,  we can
assume that for every $D_{i}$, the sequence
$$
D_{i,n}=\phi_{\tilde{\mu}_{n}}(D_{i})
$$
converges to a topological disk $E_{i}$ with respect to the
Hausdorff metric. It follows that each $E_{i}$ contains a round disk
of radius $b$ centered at $e_{i}$.  Note that by taking each $A_{i}$
thinner, we may assume that $\phi_{\tilde{\mu}_{n}}$ is univalent in
a larger disk containing $\overline{D_{i} \cup A_{i}}$ in its
interior. So by taking a subsequence if necessary, we can also
assume that
$$
A_{i,n}=\phi_{\tilde{\mu}_{n}}(A_{i})
$$
converges to a topological annulus $B_{i}$ with respect to the
Hausdorff metric. It is clear that
$$
\hbox{mod}(B_{i})=\hbox{mod}(A_{i}).
$$

Note that $\gamma_{i, n} = (\phi_{\tilde{\mu}_{n}} \circ
\Phi_{i})({\Bbb T}_{r_{i}})$. Since $(\phi_{\tilde{\mu}_{n}} \circ
\Phi_{i})$ maps $H_{i}$ univalently into ${\Bbb P}^{1} \setminus
\{0, 1, \infty\}$ and since $\tilde{\mu}_{n} \in T_{f,b}$,  it
follows again by taking a subsequence if necessary, that we may
assume that $\phi_{\tilde{\mu}_{n}} \circ \Phi_{i}$ converges to
some univalent function $\Lambda_{i}$ defined on $H_{i}$, and
moreover,
\begin{equation}\label{uniform}(\phi_{\tilde{\mu}_{n}} \circ \Phi_{i})(z)  \to
\Lambda_{i}(z)\hbox{ uniformly in any compact set of } H_{i}.
\end{equation} Let
$$ \gamma_{i} = \Lambda_{i}({\Bbb T}_{r_{i}}).$$It is not difficult to see
that every $\gamma_{i}$ is a real analytic and simple closed curve
which is homotopic to the core curve of $B_{i}$.

Again by taking a subsequence if necessary,  we may assume that as
$n \to \infty$, for every $z_{j}\in P_{1}$,
$$w_{j,n}=\phi_{\tilde{\mu}_{n}}(z_{j}) $$ converges to  some $w_{j}$ in the
spherical distance. It is important to note that the objects in
$\{E_{i}\}$ and $\{w_{j}\}$ still satisfy the bounded geometry
properties in Definition~\ref{bounded geometry}.  Let
$$
{\mathcal {R}}={\Bbb P}^{1} \setminus (\cup_{i} \overline{E}_{i}
\cup \{ w_{j}\}).
$$

Since $g_{n}(w)$ is a holomorphic function  on ${\Bbb P}^{1}
\setminus \phi_{\tilde{\mu}_{n}}(Q_{f})$, it follows that for any
compact set $W \subset {\mathcal{R}}$, the function $g_{n}(w)$ is
defined on $W$ provided $n$ is large enough. Moreover, from
(\ref{curve-1}), for any such compact set $W$, we can always take
$r_{i}$ close to $1$ or $R_{i}$ such that
$$
W \subset R_{n}.
$$

For any $w\in W$, from (\ref{decomp}) and Cauchy formula, we have
$$
g_{n}(w)=\frac{1}{2 \pi i} \int_{\cup_{i} \gamma_{i,n}}
\frac{g_{n}(\xi)}{\xi - w} d \xi
$$
$$
= \frac{1}{2 \pi i} \int_{\cup_{i} \gamma_{i,n}}
\frac{\tilde{q}_{n}(\xi)}{\xi - w} d \xi - \frac{1}{2 \pi i}
\sum_{j} \int_{\cup_{i} \gamma_{i,n}} \frac{b_{j,n}}{(\xi -
w_{j,n})(\xi - w)} d \xi
$$
Note that by assumption $\infty \in D_{1}$ and hence $\infty \notin
R_{n}$. It follows that
$$
\frac{b_{j,n}}{(\xi - w_{j,n})(\xi - w)}
$$
is holomorphic in $R_{n}$ and  the residues at the two simple poles
are equal to each other. It follows that its integral along
$\cup_{i} \gamma_{i,n}$ is zero.  We thus have
$$
g_{n}(w)= \frac{1}{2 \pi i} \int_{\cup_{i} \gamma_{i,n}}
\frac{\tilde{q}_{n}(\xi)}{\xi - w} d \xi.
$$
By (\ref{claim-1}) and the fact that $d(W, \cup_{i} \gamma_{i,n}) >
0$, it follows that $g_{n}(w)\to 0$ uniformly in $W$ as $n \to
\infty$. In particular, since $\cup_{i} \gamma_{i,n}$ is a compact
subset of $\mathcal{R}$, it follows  that $g_{n}(w) \to 0$ uniformly
for $w \in \cup_{i} \gamma_{i,n}$.  This, together with
(\ref{claim-1}) and  (\ref{decomp}),  implies
\begin{equation}\label{contr}
\int_{\cup_{i} \gamma_{i,n}} \bigg{|} \sum_{j}\frac{b_{j,n}}{w -
w_{j,n}} \bigg{|} |dw| \to 0, \quad \hbox{as} \quad n\to \infty.
\end{equation}

We claim that $b_{j,n}\to 0$ as $n\to \infty$ for each $j$. Let us
prove the claim by contradiction.  Let $\beta_{n} =\max_{j}\{
|b_{j,n}|\}$. By taking a subsequence we may assume that there is an
$\epsilon
>0$ such that $\beta_{n}\geq \epsilon$ for all $n \ge 0$.
Let $$h_{j,n}=b_{j,n}/\beta_{n}.$$ Then $\max_{j}\{ |h_{j,n}|\} =1$.
By (\ref{contr}), we have
\begin{equation}\label{final-eq}
\int_{\cup_{i} \gamma_{i,n}} \bigg{|} \sum_{j}\frac{h_{j,n}}{w -
w_{j,n}} \bigg{|} |dw| \to 0 \quad \hbox{as} \quad n\to \infty.
\end{equation}
By taking a convergent subsequence again, we may assume that every
$h_{j,n}$ converges to a number $h_{j}$ as $n$ goes to infinity. We
thus have
\begin{equation}\label{cont-1}
\max_{j}\{ |h_{j}|\}=1.
\end{equation}
From (\ref{uniform}) and (\ref{final-eq}), we have
$$
\int_{\cup_{i}\gamma_{i}} \bigg{|} \sum_{j}\frac{h_{j}}{w - w_{j}}
\bigg{|} |dw| = 0.
$$
This implies that
$$
\sum_{j}\frac{h_{j}}{w - w_{j}}=0  \hbox{ for all } w \in
\cup_{i}\gamma_{i} \hbox{ and thus equal to zero everywhere}.
$$
Since all $w_{j}$ are distinct with each other, it follows by
computing the residue at each $w_{j}$ that all  $h_{j}$ are equal to
zero. This contradicts with (\ref{cont-1}) and the claim has been
proved.

Since $g_{n}(z) \to 0$ uniformly on any compact set of $\mathcal{R}$
and $b_{j,n}\to 0$ as $n\to \infty$ for every $j$, it follows from
(\ref{decomp}) that
$$
\int_{R_{n}} |\tilde{q}_{n}(w)|dw \wedge d\overline{w}\to 0 \quad
\hbox{as} \quad n\to \infty.
$$
This, together with (\ref{con-1}), implies
$$
\int_{{\Bbb P}^{1} \setminus \phi_{\tilde{\mu}_{n}}(Q_{f})}
|\tilde{q}_{n} (w)| dw \wedge d\overline{{w}} \to 0 \quad \hbox{as}
\quad n\to \infty.
$$
This contradicts with the assumption (\ref{norm-1}) and completes
the proof of the lemma.
\end{proof}
By Propositions~\ref{key}, \ref{key'} and Lemma~\ref{expcont2},  we
have
\begin{corollary}\label{key-3}{\rm
Let $b>0$.  Then there is a constant $0< \delta < 1$ depending only
on $b$ such that
$$
\|d\sigma_{f}\big{|}_{\tau} \| \le \delta
$$
for all $\tau \in T_{f, b}$.}
\end{corollary}

 %%%%%%%%%%%%%%%%%%%%%%%%%%%%%%%%

Given any $[\mu_{0}] \in T_{f}$. Let $[\mu_{n}]=
\sigma_{f}^{n}([\mu_{0}]) = [(f^{*})^{n}\mu_{0}]$ for $n \ge 0$.

\begin{lemma}~\label{fp2}
Suppose that there exist a $b > 0$  and a point $[\mu_{0}] \in
T_{f}$ such that $\{ [\mu_{n}] \}_{n=0}^{\infty}\subset T_{f,b}$.
Then $\sigma_{f}$ has a unique fixed point in $T_{f}$.
\end{lemma}

\begin{proof}
From  Corollary~\ref{key-3} and Lemma~\ref{infinitesimal-metric}, it
follows that  $\{ [\mu_{n}]\}_{n=0}^{\infty}$ is a Cauchy sequence.
Since $ T_{f}$ is complete, $[\mu_{n}]$ converges to a limit point
$[\mu]$ in $T_{f}$, that is,
$$
\lim_{n\to\infty}[\mu_{n}]=[\mu].
$$
It follows that  $\sigma_{f}([\mu])=[\mu]$. The uniqueness of the
fixed point follows also from Corollary~\ref{key-3}.
\end{proof}

%%%%%%%%%%%%%%%%%%%%%%%%%%%%%%%%%%%%%%%%%%%%%%%%%%%%%%%%%%%%%%%%%%%%
%%%%%%%%%%%%%%%%%%%%%%%%%%%%%%%%%%%%%%%%%%%%%%%%%%%%%%%%%%%%%%%%%%%%
%%%%%%%%%%%%%%%%%%%%%%%%%%%%%%%%%%%%%%%%%%%%%%%%%%%%%%%%%%%%%%%%%%%%
%%%%%%%%%%%%%%%%%%%%%%%%%%%%%%%%%%%%%%%%%%%%%%%%%%%%%%%%%%%%%%%%%%%%

\section{No Thurston obstruction implies bounded geometry}

\begin{lemma}\label{universal-k}
Suppose that $f$ has no Thurston obstructions. Then there is an
integer  $k
> 0$ such that for every $f$-stable multi-curve $\Gamma =
 \{\gamma_{i}\}$ with $\gamma_{i}
\subset S^{2} \setminus Q_{f}$ and the associated linear
transformation matrix $A_{\Gamma}$,  we have
\begin{equation}\label{matrix-norm}
\max_{j}\sum_{i} b_{i, j} < 1/2
\end{equation}
where $A_{\Gamma}^{k} = (b_{ij})$.
\end{lemma}
\begin{proof}
Let $\Gamma =
 \{\gamma_{i}\}$  be a $f$-stable multi-curve with $\gamma_{i} \subset
S^{2} \setminus Q_{f}$. It is clear that the number of the elements
in $\Gamma$ has an upper bound which depends only on $\#(E)$. This
implies  that there can be only finitely many distinct $A_{\Gamma}$.
The lemma follows.
\end{proof}

Let $Z \subset S^{2}$ be a subset with $\#(Z) \ge 4$ and $\gamma
\subset S^{2} \setminus Z$ be a non-peripheral simple closed curve.
For $[\mu]\in T_{f}$, define
$$
w_{Z}(\gamma, [\mu]) = -\log \|\gamma\|_{\mu, Z}.
$$

By using the same argument as in the proof of Proposition 7.2 of
\cite{DH}, we have
\begin{lemma}~\label{lip}
Let $Z \subset S^{2}$ be a subset with $\#(Z) \ge 4$ and $\gamma
\subset S^{2} \setminus Z$ be a non-peripheral simple closed curve.
Then the function
$$
[\mu] \to w_{Z}(\gamma, [\mu]): T_{f}\to {\mathbb R}
$$
is  Lipschitz with Lipschitz constant $2$.
\end{lemma}
Recall that $E = P_{1}\cup \cup_{i}\{a_{i}, b_{i}\}$. Let $[\mu]\in
T_{f}$ and $b$  be a real number. Define
$$
\Gamma_{ \mu}^{b} = \{ \gamma  \;|\: \gamma \hbox{ is a } (\mu,
E)\hbox{-simple closed geodesic
 }  \hbox{ with }w_{E}(\gamma, [\mu])\geq b\},
$$
and
$$
L_{\mu} = \{ w_{E}(\gamma, [\mu])\:\:\big{|}\:\: \gamma \hbox{ is a
} (\mu, E)\hbox{-simple closed geodesic} \}.
$$

\begin{lemma}~\label{gap1}
There exists an $A > -\log\log \sqrt 2$ such that for any $[\mu] \in
T_{f}$ and any real numbers $a < b$, if
\begin{itemize}
\item[1.] $a > A$,
\item[2.]  $b-a \ge  \log d + 2d_{T}([\mu], [f^{*}\mu]) + 1$,
\item[3.] $[a, b] \cap L_{\mu} = \emptyset$,
\item[4.] $\Gamma_{\mu}^{b}\neq \emptyset$,
\end{itemize}
 then $\Gamma_{\mu}^{b}$ is a $f$-stable multi-curve in $S^{2} \setminus Q_{f}$.
\end{lemma}
\begin{proof}
Let $\gamma \in \Gamma_{\mu}^{b}$. By the first assertion of
Lemma~\ref{collar}, $\gamma$ is a non-peripheral and simple closed
curve in $S^{2} \setminus Q_{f}$ provided that $A$ is big and thus
$\|\gamma\|_{\mu, E}$ is small.  By the second assertion of
Lemma~\ref{collar}, we have
$$
w_{Q_{f}}(\gamma, [\mu]) > w_{E}(\gamma, [\mu]) - 1
$$
provided that $A$ is big and thus $\|\gamma\|_{\mu, E}$ is small.
Now suppose that $\gamma'$ is a non-peripheral component of
$f^{-1}(\gamma)$. Since $f$ is a degree $d$ branched covering map of
the sphere, it follows that
$$
w_{f^{-1}(Q_{f})}(\gamma', [f^{*}\mu]) \ge  w_{Q_{f}}(\gamma, [\mu])
- \log d.
$$
Since $E \subset f^{-1}(Q_{f})$, it follows that
$$
w_{E}(\gamma',[f^{*}\mu]) > w_{f^{-1}(Q_{f})}(\gamma', [f^{*}\mu]).
$$
By Lemma~\ref{lip}, we have
$$
w_{E}(\gamma', [\mu]) \ge w_{E}(\gamma', [f^{*}\mu]) - 2d_{T}([\mu],
[f^{*}\mu]).
$$
This implies that
$$
w_{E}(\gamma, [\mu]) - w_{E}(\gamma', [\mu]) < \log d +
2d_{T}([\mu], [f^{*}\mu]) + 1 \le  b-a.
$$
Since $w_{E}(\gamma, [\mu]) > b$ and $[a, b] \cap L_{\mu} =
\emptyset$, it follows that $w_{E}(\gamma', [\mu]) > b$. This
implies that $\gamma'$ is homotopic to some element in
$\Gamma_{\mu}^{b}$. The Lemma follows.
\end{proof}

Let $k \ge  0$ be the integer in Lemma~\ref{universal-k}.  Let
\begin{equation}\label{P-2}
P_{2} =   E \cup f^{k}(E) \cup \bigcup_{1\le j \le k}
f^{j}(\Omega_{f}).
\end{equation}
\begin{lemma}\label{comle-le}
There exists an $\epsilon_{0}> 0$ which is independent of $\mu$ such
that for any $(\mu, P_{2})$-simple closed geodesic $\eta$, if
$\|\eta\|_{\mu, P_{2}} \le \epsilon_{0}$, then there is a $(\mu,
E)$-simple closed geodesic $\gamma$ such that $\eta$ is homotopic to
$\gamma$ in $S^{2}\setminus P_{2}$.
\end{lemma}
\begin{proof}
Suppose $\eta$ is not homotopic to any  $(\mu, E)$-simple closed
geodesic in $S^{2}\setminus P_{2}$. Then there is at least one
holomorphic disk $D_{i}$, such that $\gamma$  separate the points in
$D_{i} \cap P_{2}$. Let $x, y \in  D_{i} \cap P_{2}$ which are
separated by $\gamma$. Let $z \in P_{2} \setminus \overline{D_{i}}$.
Let $\phi:S^{2} \to {\Bbb P}^{1}$ be the homeomorphism which solves
the Beltrami equation given by $\mu$  and which maps $x$, $y$, and
$z$ respectively to $0$, $1$, and the infinity. Then there are two
cases.

In the first case,  $\phi(\gamma)$ is contained in
$\phi(\overline{D_{i}}\cup A_{i})$.  Note that $A_{i}$ is the
shielding ring attached to the outside of $D_{i}$.  Then
$\phi(\gamma)$ must enclose the $\phi$-images of at least two points
in $\overline{D_{i}}\cap P_{2}$.  Since $\phi$ is univalent in
$\overline{D_{i}}\cup A_{i})$, it follows from Koebe's distortion
theorem that there exist $R> 0$ and $D > 0$ independent of $\mu$
such that $\phi(\gamma) \cap \{z\:|z| \le R\} \ne \emptyset$ and the
Euclidean diameter of $\phi(\gamma)$ is greater than $D$. This,
together with Koebe's distortion theorem,  implies that the
hyperbolic length of $\phi(\gamma)$ in ${\Bbb P}^{1} \setminus \{0,
1, \infty\}$, and thus $\|\gamma\|_{\mu, P_{2}}$,  has a positive
lower bound independent of $\mu$.

In the second case, $\phi(\gamma)$ is not contained in
$\phi(\overline{D_{i}}\cup A_{i})$. Since $\phi(\gamma)$ separates
$\phi(x)$ and $\phi(y)$, it follows that $\phi(\gamma)$ must cross
through the annulus $\phi(A_{i})$. By Koebe's distortion theorem,
the annulus $\phi(A_{i})$ has definite thickness. This again implies
that there exist $R> 0$ and $D > 0$ independent of $\mu$ such that
$\phi(\gamma) \cap \{z\:|z| \le R\} \ne \emptyset$ and the Euclidean
diameter of $\phi(\gamma)$ is greater than $D$. Therefore, the
hyperbolic length of $\phi(\gamma)$ in ${\Bbb P}^{1} \setminus \{0,
1, \infty\}$, and thus $\|\gamma\|_{\mu, P_{2}}$,  has a positive
lower bound independent of $\mu$. The proof of the lemma is
completed.
\end{proof}

Note  that
\begin{equation}\label{covering}
f^{k}: S^{2} \setminus f^{-k}(P_{2}) \to S^{2} \setminus P_{2}
\end{equation}
is a covering map of degree $d^{k}$. Let $A > -\log\log \sqrt 2$ be
the constant in Lemma~\ref{gap1}.
\begin{lemma}~\label{ineq'}
Let   $B > A$. Then there exists a constant $M>0$ depending only on
the numbers $k$, $B$,  $\#(E)$, $\epsilon_{0}$, and the degree $d$
of $f$, such that for any $[\mu] \in T_{f}$ and any real numbers $a
< b$, if
\begin{itemize}
\item[1.] $A < a < B$,
\item[2.] $b -a \ge \log d + 2 d_{T}([\mu], [f^{*}\mu]) + 1$,
\item[3.] $[a,  b]\cap L_{\mu} = \emptyset$,
\item[4.] $\Gamma_{\mu}^{b}\ne \emptyset$,
\end{itemize}
then
$$
 \sum_{\gamma\in \Gamma_{\mu}^{b}}
\frac{1}{\|\gamma\|_{\nu, E}} \le \frac{1}{2} \sum_{\gamma\in
\Gamma_{\mu}^{b}} \frac{1}{\|\gamma\|_{\mu, E}}  + M,
$$
where $\nu = ({f^{k}})^{*}(\mu)$ and $k \ge 0$ is the integer in
Lemma~\ref{universal-k}.
\end{lemma}

\begin{proof}

By Lemma~\ref{gap1},  $\Gamma_{\mu}^{b}$ is a $f$-stable multi-curve
in $S^{2} \setminus Q_{f}$. For each $\gamma_{j}\in
\Gamma_{\mu}^{b}$, let $\gamma_{i,j,\alpha}$ be any component of
$f^{-k}(\gamma_{j})$ homotopic to $\gamma_{i}$ in $S^{2} \setminus
Q_{f}$. Then $\gamma_{i,j,\alpha}$ is also homotopic to $\gamma_{i}$
in $S^{2} \setminus E$.

Let $g = \phi_{\mu} \circ f^{k} \circ \phi_{\nu}^{-1}$. Then $g$ is
a rational map. It follows from (\ref{covering}) that
$$
g: {\mathbb P}^{1}\setminus\phi_{\nu} (f^{-k}(P_{2}))\to {\mathbb
P}^{1}\setminus\phi_{\mu}(P_{2})
$$
is a holomorphic covering map, and therefore,
$$
\|\gamma_{i,j,\alpha}\|_{\nu, f^{-k}(P_{2})} = d_{i,j,\alpha}
\|\gamma_{j}\|_{\mu, P_{2}}
$$
where $d_{i,j,\alpha}\leq d^{k}$ is the degree of
$$
f^{k}: \gamma_{i,j,\alpha}\to \gamma_{j}.
$$
Thus
$$
\sum_{\alpha} \frac{1}{\|\gamma_{i,j,\alpha}\|_{\nu, f^{-k}(P_{2})}}
= \Big( \sum_{\alpha} \frac{1}{d_{i,j,\alpha}}\Big) \frac{1}{
\|\gamma_{j}\|_{\mu, P_{2}}} =b_{ij} \frac{1}{ \|\gamma_{j}\|_{\mu,
P_{2}}}
$$

Since $E \subset P_{2}$ by (\ref{P-2}), it follows that
$\|\gamma_{j}\|_{\mu, P_{2}} > \|\gamma_{j}\|_{\mu,E}$, and
therefore
$$
\frac{1}{\|\gamma_{j}\|_{\mu, P_{2}}} <
\frac{1}{\|\gamma_{j}\|_{\mu,E}}.
$$
This implies
\begin{equation}\label{imme-in}
\sum_{\alpha} \frac{1}{\|\gamma_{i,j,\alpha}\|_{\nu, f^{-k}(P_{2})}}
< b_{ij} \frac{1}{ \|\gamma_{j}\|_{\mu,E}}
\end{equation}

Note that $E \subset f^{-k}(P_{2})$ by (\ref{P-2}). Let $p$ denote
the number of the points in $f^{-k}(P_{2})\setminus E$. It follows
from (\ref{P-2}) that there is  a constant $$0< C(k, d, \#(E))<
\infty$$ depending only on $d$, $k$, and $\#(E)$ such that $ p \le
C(k, d, \#(E))$.

Now we claim that for any $(\nu, f^{-k}(P_{2}))$-simple closed
geodesic $\gamma$ which is homotopic to $\gamma_{i}$ in $S^{2}
\setminus E$, either $\gamma$ is homotopic to  some $\gamma_{i, j,
\alpha}$ in $S^{2} \setminus f^{-k}(P_{2})$, or
$$\|\gamma\|_{\nu, f^{-k}(P_{2})} >  \min\{e^{-B}, \epsilon_{0}\}.$$
Let us prove the claim. In fact, if $\gamma$ is not homotopic in
$S^{2} \setminus f^{-k}(P_{2})$ to some $\gamma_{i, j, \alpha}$,
then $f^{k}(\gamma)$ is a $(\mu, P_{2})$-simple closed geodesic
which is not homotopic to any $\gamma_{j}$ in $S^{2} \setminus
P_{2}$.  There are two cases. In the first case, $f^{k}(\gamma)$ is
 homotopic in $S^{2} \setminus P_{2}$ to some $(\mu, E)$-simple closed geodesic $\xi$
 which does not belong to $\Gamma_{\mu}^{b}$.
 By the assumption that $L_{\mu} \cap [a, b] = \emptyset$, we have
 $$
  \|f^{k}(\gamma)\|_{\mu, P_{2}} > \|f^{k}(\gamma)\|_{\mu, E} = \|\xi\|_{\mu, E} > e^{-a} > e^{-B}.
 $$
In the second case, $f^{k}(\gamma)$ is not homotopic in $S^{2}
\setminus P_{2}$ to any  $(\mu, E)$-simple closed geodesic. By
Lemma~\ref{comle-le}, we have
$$
 \|f^{k}(\gamma)\|_{\mu, P_{2}} > \epsilon_{0}.
$$
We thus have
$$
\|\gamma\|_{\nu, f^{-k}(P_{2})} \ge  \|f^{k}(\gamma)\|_{\mu, P_{2}}
> \min\{e^{-B}, \epsilon_{0}\}.
$$

Now  from the left hand of the inequality given by (c) in Theorem
7.1 of \cite{DH}, we have
$$
\frac{1}{\| \gamma_{i}\|_{\nu,E}} \leq \sum_{j}\sum_{\alpha}
\frac{1}{\|\gamma_{i,j,\alpha}\|_{\nu, f^{-k}(P_{2})}} +
\frac{2}{\pi} + \frac{C(k, d, \#(E))+1}{\min\{e^{-B},
\epsilon_{0}\}}.
$$
Let $$ M' = \frac{2}{\pi} + \frac{C(k, d, \#(E))+1}{\min\{e^{-B},
\epsilon_{0}\}}.$$ Thus
$$
 \sum_{\gamma\in \Gamma_{\mu}^{b}}
\frac{1}{\|\gamma\|_{\nu, E}} \leq \sum_{i}\sum_{j}\sum_{\alpha}
\frac{1}{\|\gamma_{i,j,\alpha}\|_{\nu, f^{-k}(P_{2})}} + KM'.
$$
where $K$ is the number of the curves in $\Gamma$ which is bounded
above by $\#(E)-3$. Let $$M =(\#(E)-3) M'.$$ By (\ref{imme-in}), we
have
$$
\sum_{\gamma\in \Gamma_{\mu}^{b}} \frac{1}{\|\gamma\|_{\nu, E}} \leq
\sum_{j} \Big( \sum_{i} b_{ij}\Big) \frac{1}{
\|\gamma_{j}\|_{\mu,E}} +M \leq \frac{1}{2}\sum_{\gamma\in
\Gamma_{\mu}^{b}} \frac{1}{\|\gamma\|_{\mu, E}} +M.
$$
This completes the proof of the Lemma.
\end{proof}

The following is a technical lemma from Calculus.

\begin{lemma}~\label{cal}
Let  $b_{0}> 1$, $c_{0}, M_{0}> 0$, and  integer $m_{0}> 1$ be
given. Then for any sequence $\{x_{n}\}_{n=0}^{\infty}$ of positive
numbers satisfying
\begin{enumerate}
\item $x_{0}\leq c_{0}$,
\item $x_{n+1}/x_{n}\leq b_{0}$,
\item if $x_{n}\geq M_{0}$, then $x_{n+m_{0}}\leq x_{n}$,
\end{enumerate}
one has
$$
x_{n} \le \max\{ b_{0}^{m_{0}-1}c_{0}, b_{0}^{m_{0}}M_{0}\}, \quad
\forall n\geq 0.
$$
\end{lemma}

\begin{proof}
Let $C = \max\{ b_{0}^{m_{0}-1}c_{0}, b_{0}^{m_{0}}M_{0}\}$.   It is
sufficient to prove that $$x_{i+ lm_{0}} \le C$$ for all $0 \le i
\le m_{0} -1$ and $l \ge 0$.  Take an arbitrary integer $0 \le i \le
m_{0} -1$.  Let us prove that $$x_{i+ lm_{0}} \le C$$ for all $l \ge
0$ by induction. For $l = 0$, we have
$$
x_{i} \le  b_{0}^{i}x_{0} \le  b_{0}^{m_{0} -1}c_{0} \le  C.
$$
Now assume that
\begin{equation}\label{assump}
x_{i+ km_{0}} \le C
\end{equation}
for some integer $k \ge 0$. Let us prove that
$$
x_{i+ (k+1)m_{0}} \le  C.
$$
In fact, there are two cases by assumption (\ref{assump}). In the
first case, $x_{i + km_{0}} <  M$. In this case, we have
$$
x_{i+ (k+1)m_{0}} \le  b_{0}^{m_{0}} x_{i + km_{0}}  < b_{0}^{m_{0}}
M \le  C.
$$
In the second case, $x_{i + km_{0}} \ge M$. Then we have
$$
x_{i+ (k+1)m_{0}}  \le x_{i + km_{0}}  \le  C.
$$
This proves that $x_{i+ lm_{0}} \le C$ for all $l \ge 0$. Since this
holds for any $0 \le i \le m_{0} -1$, the lemma follows.
\end{proof}

\begin{lemma}~\label{bg'} If $f$ has no Thurston obstructions,
then for any $[\mu_{0}]\in T_{f}$, there exists a constant $b>0$
such that for all $n \ge 1$,
$$
[\mu_{n}] \in  T_{f, b},
$$
where $\mu_{n} = (f^{*})^{n}(\mu_{0})$.
\end{lemma}

\begin{proof}
Since $f$ is holomorphic on $\cup A_{i}$ and $f(\cup A_{i}) \subset
\cup D_{i}$, it follows that  for all $n \ge 1$, $\mu_{n}(z) = 0$ on
$\cup A_{i}$. By Lemma~\ref{gdb1}, it is equivalent to prove that
there is a uniform positive lower bound of the length of all the
$(\mu_{n}, E)$-simple closed geodesics.

Let $D=d_{T}([\mu_{0}], [\mu_{1}])$. Then by
Lemma~\ref{infinitesimal-metric} and  Corollary~\ref{key-3}, we have
$$
d_{T}([\mu_{n}], [\mu_{n+1}])\leq D \quad \hbox{for all}\:  n \geq
0.
$$

Let $K = \#(E) -3 \ge 1$ and  $k \ge 1$ be the integer in
Lemma~\ref{universal-k}.  Let $l_{0}\geq 1$ be the least integer
such that
\begin{equation}\label{K}
K < 2^{l_{0}-1}.
\end{equation}
Now it is sufficient to prove that there exist positive constants
$c_{0}, M_{0} > 0$,  $b_{0} > 1$, and an integer $m_{0} > 0$,  such
that
 the sequence
$$
x_{n} = \max_{\gamma}\{\|\gamma\|_{\mu_{n}, E}^{-1}\},
$$
where $\max$ is taken  over all the $(\mu_{n}, E)$-simple closed
geodesics,  satisfies the three conditions in Lemma~\ref{cal}.

By Corollary 6.6 of \cite{DH},  there are at most $K$  $(\mu_{n},
E)$-simple closed geodesics which has hyperbolic length less than
$\log(\sqrt 2 +1)$.  This implies that we can have $c_{0}>0$ such
that
$$
x_{0} \leq c_{0}.
$$
It is the first condition in Lemma~\ref{cal}. From Lemma~\ref{lip}
we can take $b_{0}=e^{2D}$.

Recall that we use $d$ to denote the degree of $f$.   Let  $k_{0} =
\log d + 2 D$ and   $m_{0} = kl_{0}$. Let
\begin{equation}\label{k1}
k_{1} = k_{0} + 4 m_{0} D +1 .
\end{equation}
In particular,
$k_{1}
> \log d + 2 D +1$. Let $A
> -\log \log(\sqrt 2 + 1)$ be the constant in Lemma~\ref{gap1}. In
Lemma~\ref{ineq'}, take
$$
B = A + (K +1)k_{1}
$$
and let $M$ denote the corresponding constant there. Let
$$
M_{0} = \max\{e^{B}, 2^{l_{0} + 1} M\}.
$$

It  remains to prove that if $x_{n}
> M_{0}$, then $x_{n+m_{0}} <  x_{n}$.  To see this, suppose that
$x_{n} > M_{0}$. It follows that there is a $(\mu_{n}, E)$-simple
close geodesic such that $w_{E}(\gamma, [\mu_{n}])
> B$. Then by the choice of the numbers $k_{1}$ and $B$, and the fact that there are at
most $K$ $(\mu_{n}, E)$-simple closed geodesics which have
hyperbolic length less than $\log(\sqrt 2 +1)$, one can take an
interval $[a, b]$ such that
\begin{itemize}
 \item[1.] $A < a < b< B$,
 \item[2.]  $b - a = k_{1}$,
 \item[3.]  $[a, b] \cap L_{\mu_{n}} = \emptyset$.
 \end{itemize}

It follows that $\Gamma_{\mu_{n}}^{b} \ne \emptyset$ and therefore
is a $f-$stable multicurve by Lemma~\ref{gap1}. Now for each $i = 0,
1, \cdots, l_{0}$, let
$$
[a_{i}, b_{i}] = [a + 2kiD, b - 2ki D].
$$
By Lemma~\ref{lip},  the gap condition $b - a = k_{1}$, and
(\ref{k1}),  it follows that each family
$\Gamma_{\mu_{n+ki}}^{b_{i}}$,  $0\le i \le l_{0}$,  contains the
same set of homotopy classes of simple closed curves as
$\Gamma_{\mu_{n}}^{b}$. Let us simply denote each of  them by
$\Gamma$. Now for each $0 \le i \le l_{0}-1$, let $\mu = \mu_{n+ki}$
and $\nu = \mu_{n+k(i+1)}$, and let $[a_{i}, b_{i}]$ be the
corresponding gap interval. Then the conditions in Lemma~\ref{ineq'}
are satisfied with the constants $A$ and $B$ given as above. By
Lemma~\ref{ineq'}, we have
$$
 \sum_{\gamma\in \Gamma}
\frac{1}{\|\gamma\|_{\mu_{n+k(i+1)}, E}}  \leq \frac{1}{2}
\sum_{\gamma\in \Gamma} \frac{1}{\|\gamma\|_{\mu_{n+ki}, E}}  + M
$$
for $0 \le i \le l_{0}-1$. It follows from $m_{0} = kl_{0}$ that
\begin{equation}\label{equa-f1}
 \sum_{\gamma\in \Gamma}
\frac{1}{\|\gamma\|_{\mu_{n+m_{0}}, E}}  \leq \frac{1}{2^{l_{0}}}
 \sum_{\gamma\in \Gamma}
\frac{1}{\|\gamma\|_{\mu_{n}, E}}  + 2M.
\end{equation}

Since
$$
 \sum_{\gamma\in \Gamma}
\frac{1}{\|\gamma\|_{\mu_{n}, E}} \ge  x_{n}
> M_{0} \ge 2^{l_{0} + 1} M,$$ it follows that
\begin{equation}\label{equa-f2}
M < \frac{1}{2^{l_{0}+1}} \sum_{\gamma\in \Gamma}
\frac{1}{\|\gamma\|_{\mu_{n}, E}}.
\end{equation}
From (\ref{equa-f1}) and (\ref{equa-f2}), we have
\begin{equation}\label{final}
 \sum_{\gamma\in \Gamma}
\frac{1}{\|\gamma\|_{\mu_{n+m_{0}}, E}}  < \frac{1}{2^{l_{0}-1}}
 \sum_{\gamma\in \Gamma}
\frac{1}{\|\gamma\|_{\mu_{n}, E}} .
\end{equation}
Since the number of the elements in $\Gamma$ is at most $K$, it
follows that
$$
 \sum_{\gamma\in \Gamma}
\frac{1}{\|\gamma\|_{\mu_{n}, E}}  \le K x_{n}.
$$
From (\ref{K}) and (\ref{final}),  we have
$$
x_{n+m_{0}} \le   \sum_{\gamma\in \Gamma}
\frac{1}{\|\gamma\|_{\mu_{n+m_{0}}, E}} < \frac{1}{2^{l_{0}-1}}
 \sum_{\gamma\in \Gamma}
\frac{1}{\|\gamma\|_{\mu_{n}, E}} \le \frac{K}{2^{l_{0}-1}} x_{n} <
x_{n}.
$$

\end{proof}

The Main Theorem now follows from Lemmas~\ref{fp},~\ref{fp2}, and
\ref{bg'}.

\bibliographystyle{amsalpha}

\end{document}